\documentclass[12pt]{article}

\usepackage[T1]{fontenc}
\usepackage{amsmath,amsthm,mathtools}

\usepackage[dvipsnames]{xcolor}
\usepackage{newtxtext}
\usepackage{newtxmath}
\usepackage{comment}

\usepackage{iftex}
\ifPDFTeX\usepackage[utf8]{inputenc}\fi
\usepackage[a4paper,margin=1in]{geometry}
\usepackage{aliascnt}
\usepackage{enumitem}
\usepackage[protrusion=true,expansion=false]{microtype}
\usepackage[hyperfootnotes=false,colorlinks=true,linkcolor=blue,citecolor=blue,
  urlcolor=blue]{hyperref}

\makeatletter
\newcommand\blfootnote[1]{
  \begingroup
  \renewcommand\@makefntext[1]{\noindent\strut ##1}
  \renewcommand\thefootnote{}\footnotetext{#1}
  \endgroup
}
\makeatother

\setlist{itemsep=0.25em,topsep=0.4em}
\newtheorem{theorem}{Theorem}[section]
\newaliascnt{proposition}{theorem}
\newtheorem{proposition}[proposition]{Proposition}
\aliascntresetthe{proposition}
\newaliascnt{lemma}{theorem}
\newtheorem{lemma}[lemma]{Lemma}
\aliascntresetthe{lemma}
\newaliascnt{question}{theorem}
\newtheorem{question}[question]{Question}
\aliascntresetthe{question}
\newaliascnt{corollary}{theorem}
\newtheorem{corollary}[corollary]{Corollary}
\aliascntresetthe{corollary}
\theoremstyle{definition}
\newaliascnt{definition}{theorem}
\newtheorem{definition}[definition]{Definition}
\aliascntresetthe{definition}
\newaliascnt{remark}{theorem}
\newtheorem{remark}[remark]{Remark}
\aliascntresetthe{remark}
\newaliascnt{example}{theorem}
\newtheorem{example}[example]{Example}
\aliascntresetthe{example}

\newtheorem{citedtheorem}{Theorem}

\usepackage[capitalise,noabbrev]{cleveref}

\crefname{citedtheorem}{Theorem}{Theorems}
\Crefname{citedtheorem}{Theorem}{Theorems}
\crefname{question}{Question}{Questions}
\Crefname{question}{Question}{Questions}
\newcommand{\R}{\mathbb R}
\newcommand{\C}{\mathbb C}
\newcommand{\HH}{\mathbb H}
\newcommand{\F}{\mathbb F}
\newcommand{\eps}{\varepsilon}
\newcommand{\ip}[2]{\langle #1,#2\rangle}
\newcommand{\h}{\mathfrak h}
\newcommand{\kk}{\mathfrak k}
\newcommand{\gfr}{\mathfrak g}
\newcommand{\son}{\mathfrak{so}}
\newcommand{\sun}{\mathfrak{su}}
\newcommand{\spn}{\mathfrak{sp}}
\newcommand{\dd}{\,\mathrm d}
\DeclareMathOperator{\diam}{diam}
\DeclareMathOperator{\Gr}{Gr}
\DeclareMathOperator{\Isom}{Isom}
\DeclareMathOperator{\tr}{tr}
\DeclareMathOperator{\Rea}{Re}
\DeclareMathOperator{\pr}{pr}
\DeclareMathOperator{\Hom}{Hom}
\DeclareMathOperator{\Lie}{Lie}
\DeclareMathOperator{\Ad}{Ad}
\DeclareMathOperator{\Sym}{Sym}
\numberwithin{equation}{section}

\title{\textbf{A diameter gap for irreducible compact symmetric spaces of fixed rank}}
\author{Zhiqi Chen, Shaoqiang Deng and Hui Zhang}
\date{}
\hypersetup{pdftitle={A diameter gap for irreducible compact symmetric spaces},
  pdfauthor={Zhiqi Chen, Shaoqiang Deng and Hui Zhang}}

\begin{document}
\raggedbottom
\maketitle

\begin{abstract}
We prove that, for every positive integer $r$, the positive normalized
diameters of isometric quotients of simply connected irreducible compact
Riemannian symmetric spaces of rank $r$ have a uniform positive lower
bound. The bound is independent of the dimension and of the acting
group, which need not be connected. For the higher-dimensional
Grassmannian families, we obtain an explicit bound for
positive-dimensional acting groups by constructing nonconstant invariant
functions of bounded trigonometric degree. The real and complex cases
use moment maps. The quaternionic case follows from a quartic rigidity
theorem proved by means of invariant four-forms and the classification
of compact quaternionic symmetric spaces. We also give examples showing
why the rank and the relative scales of product factors must be controlled.
\end{abstract}


\section{Introduction}\label{sec:introduction}

Let $M$ be a compact connected Riemannian manifold of positive dimension,
and let $G\subset\Isom(M)$ be a subgroup. 
The quotient metric and diameter  are respectively defined  by 
\begin{align*}
d_{M/G}(Gp,Gq)=\inf_{g\in G}d_M(p,gq),\qquad
 \diam(M/G):=\sup_{p,q\in M}d_{M/G}(Gp,Gq).
\end{align*}
Note that the diameter of the quotient space does not change if the  acting group $G\subset\Isom(M)$ is replaced by its closure. We may therefore work with compact groups.   Define  the scale-invariant quantity
\begin{align*}
 \Delta(M,G)=\frac{\diam(M/G)}{\diam M}.
\end{align*}
Our main result is the following.

\begin{comment}
Let $M$ be a compact connected Riemannian manifold of positive dimension,
and let $G\subset\Isom(M)$ be a subgroup. We write $M/G$ for the
metric quotient by the closure $\overline G$. Thus
\[
 d_{M/G}(\overline Gp,\overline Gq)
   =\inf_{g\in G}d_M(p,gq).
\]
Passing from $G$ to its closure does not change this infimum. For a
nonclosed group, this convention identifies raw orbits at distance zero.
We shall use the scale-invariant quantity
\[
 \Delta(M,G)=\frac{\diam(M/G)}{\diam M}.
\]
Our main result is the following.
\end{comment}

\begin{theorem}\label{thm:main}
For every positive integer $r$, there exists $\eps_r>0$ such that,
for every simply connected irreducible compact Riemannian symmetric
space $M$ of rank $r$ and every subgroup $G\subset\Isom(M)$,
\[
 \Delta(M,G)=0\qquad\text{or}\qquad\Delta(M,G)\ge\eps_r.
\]
\end{theorem}

The zero-diameter case is equivalent to transitivity of $\overline G$.
In particular, the theorem applies to disconnected groups and to groups
with arbitrarily many components. The proof does not give a general
explicit formula for $\eps_r$; an explicit estimate for the
Grassmannian families is given in \cref{thm:tail}.

Diameter bounds for spherical quotients were studied in
\cite{McGowan,Greenwald,McGowanSearle,DGMS,Green}.
The following two results of Gorodski, Lange, Lytchak, and Mendes
\cite{GLLM} will be used in the proof.

\begin{citedtheorem}[\cite{GLLM}]\label{thm:sphere-gap}
There exists $\eps>0$ such that, for every $n\ge2$ and every subgroup
$G\subset\Isom(S^n)$ of the unit sphere,
\[
 \diam(S^n/G)=0\qquad\text{or}\qquad\diam(S^n/G)\ge\eps.
\]
\end{citedtheorem}

\begin{citedtheorem}[\cite{GLLM}]\label{thm:fixed-space}
Let $M$ be a compact connected homogeneous Riemannian manifold. Then $\pi_1(M)$
is finite if and only if there exists $\eps_M>0$ such that every
subgroup $G\subset\Isom(M)$ satisfies
\[
 \diam(M/G)=0\qquad\text{or}\qquad\diam(M/G)>\eps_M.
\]
\end{citedtheorem}

In \cite[Question~2]{GLLM}, the authors ask whether quotients of
simply connected compact symmetric spaces admit a diameter bound
depending only on the rank. With the diameter of the original space
fixed, this takes the following form.

\begin{question}\label{question:Q}
Is there a positive lower bound for $\Delta(M,G)$, whenever it is
nonzero, depending only on the rank of the simply connected compact
symmetric space $M$?
\end{question}

\Cref{thm:main} gives an affirmative answer for irreducible spaces.
The rank-one case already follows from the two cited theorems, as
observed in \cite{GLLM}. Indeed, spheres are covered by
\cref{thm:sphere-gap}. A quotient of a complex or quaternionic
projective space is isometric to a sphere quotient by a compact
extension of the acting group that contains the Hopf fibre action.
We give the precise reduction in the proof of \cref{thm:main}.
The Cayley plane is a single homothety class and is covered by
\cref{thm:fixed-space}.

Neither the rank restriction nor the irreducibility hypothesis can be
removed from the statement as given. In \cref{ex:rank}, the rank tends
to infinity while the normalized quotient diameter tends to zero.
In \cref{ex:products}, independent rescaling of two spherical factors
gives the same phenomenon at rank two.

We also obtain the following consequence for spaces with an
irreducible compact symmetric universal cover.

\begin{corollary}\label{cor:cover}
Let $M$ be a compact connected Riemannian manifold whose Riemannian
universal cover is an irreducible compact symmetric space of rank $r$.
For every subgroup $G\subset\Isom(M)$, either $\Delta(M,G)=0$ or
$\Delta(M,G)\ge\eps_r$, with the same constant as in \cref{thm:main}.
\end{corollary}

For fixed rank $r$, the classification reduces the proof to finitely
many homothety classes and the three Grassmannian families
\begin{equation}\label{eq:families}
 X_{\R,r,N}=\Gr_r^+(\R^N),\qquad
 X_{\C,r,N}=\Gr_r(\C^N),\qquad
 X_{\HH,r,N}=\Gr_r(\HH^N),\qquad N\ge2r.
\end{equation}
Here the superscript $+$ denotes oriented real planes. The map
\begin{equation}\label{eq:forgetmap}
 \pi:\Gr_r^+(\R^N)\longrightarrow\Gr_r(\R^N)
\end{equation}
forgets the orientation and is a two-sheeted Riemannian covering.
The quantitative estimate needed for positive-dimensional groups is
as follows.

\begin{theorem}\label{thm:tail}
Let $N>\max\{28,2r\}$. Assume either $\F=\R$ and $r\ge2$, or
$\F\in\{\C,\HH\}$ and $r\ge1$. If
$G\subset\Isom(X_{\F,r,N})$ is compact and nontransitive, with
$\dim G^0\ge1$, then
\begin{equation}\label{eq:tail}
 \Delta(X_{\F,r,N},G)\ge\frac1{\pi r}.
\end{equation}
\end{theorem}

For a subgroup that is not closed, \cref{thm:tail} applies whenever
its closure is nontransitive and has positive-dimensional identity
component, since passage to the closure preserves the quotient metric.
The same numerical bound fails for the real rank-one family; see
\cref{ex:rank-one-tail}. Finite groups are treated separately using
Green's estimate for matrix coefficients \cite{Green}.

The proof of \cref{thm:tail} uses a nonconstant invariant function
whose restriction to each rank-one Grassmannian rotation has bounded
trigonometric degree. Bernstein's inequality then bounds its gradient
in terms of that degree and $r$. After normalizing the range of the
function, this gives a lower bound for the quotient diameter. The
main point is to construct such functions and prove their
nonconstancy. We use moment maps in the real and complex cases.
In the quaternionic case, a constant quartic function yields an
invariant four-form identity, from which we construct a compact
symmetric pair and apply Wolf's classification.

In Section~\ref{sec:reduction}, we record the classification and
isometry-group facts used in the reduction. Section~\ref{sec:analytic}
contains the metric and analytic estimates. Finite groups are treated
in Section~\ref{sec:finite}, and the three positive-dimensional cases
in Section~\ref{sec:real-complex}. We prove \cref{thm:main,cor:cover}
and give the examples in Section~\ref{sec:assembly}.

\section{Preliminaries}\label{sec:reduction}
We first record the classification facts that reduce the proof to
Grassmannians. We then describe the ambient matrix groups used to lift
isometry actions.
\subsection{Simply connected
irreducible compact symmetric spaces}
The classification and rank formulas of simply connected
irreducible compact symmetric spaces are given in
\cite{Gorodski,Helgason}. We first list the classical spaces.

\begin{center}
\small
\renewcommand{\arraystretch}{1.25}
\begin{tabular}{lllll}
\hline
Type & Space & Dimension & Rank & Parameters\\\hline
AI & $SU(n)/SO(n)$
 & $(n-1)(n+2)/2$ & $n-1$ & $n\ge2$\\
AII & $SU(2n)/Sp(n)$
 & $(n-1)(2n+1)$ & $n-1$ & $n\ge2$\\
AIII & $SU(p+q)/S(U(p)\times U(q))$
 & $2pq$ & $p$ & $1\le p\le q$\\
BDI & $SO(p+q)/(SO(p)\times SO(q))$
 & $pq$ & $p$ & $1\le p\le q$\\
CI & $Sp(n)/U(n)$
 & $n(n+1)$ & $n$ & $n\ge1$\\
CII & $Sp(p+q)/(Sp(p)\times Sp(q))$
 & $4pq$ & $p$ & $1\le p\le q$\\
DIII & $SO(2n)/U(n)$
 & $n(n-1)$ & $\lfloor n/2\rfloor$ & $n\ge2$.\\
 \hline
\end{tabular}
\end{center}
\begin{remark}
In the BDI row we exclude $(p,q)=(1,1)$, which gives a circle,
and $(p,q)=(2,2)$, which gives the reducible space $S^2\times S^2$.
The remaining BDI spaces are the oriented real Grassmannians
$\Gr_p^+(\R^{p+q})$; in particular, $p=1$ gives the sphere $S^q$. The AIII and CII rows are respectively
$\Gr_p(\C^{p+q})$ and $\Gr_p(\HH^{p+q})$.
\end{remark}

For the exceptional spaces, it is convenient to specify the standard
Cartan pair by its compact group $G$ and isotropy Lie algebra $\kk$.
Here, $G$ denotes the simply connected compact form, $\sigma$ is the
corresponding Cartan involution, and
$K=(G^\sigma)^0$ is the connected subgroup with Lie algebra $\kk$.
The space in each row is $G/K$.

\begin{center}
\small
\renewcommand{\arraystretch}{1.2}
\begin{tabular}{lllll}
\hline
Type & $G$ & $\kk$ & Dimension & Rank\\\hline
EI & $E_6$ & $\spn(4)$ & $42$ & $6$\\
EII & $E_6$ & $\sun(6)\oplus\spn(1)$ & $40$ & $4$\\
EIII & $E_6$ & $\son(10)\oplus\mathfrak u(1)$ & $32$ & $2$\\
EIV & $E_6$ & $\mathfrak f_4$ & $26$ & $2$\\
EV & $E_7$ & $\sun(8)$ & $70$ & $7$\\
EVI & $E_7$ & $\son(12)\oplus\spn(1)$ & $64$ & $4$\\
EVII & $E_7$ & $\mathfrak e_6\oplus\mathfrak u(1)$ & $54$ & $3$\\
EVIII & $E_8$ & $\son(16)$ & $128$ & $8$\\
EIX & $E_8$ & $\mathfrak e_7\oplus\spn(1)$ & $112$ & $4$\\
FI & $F_4$ & $\spn(3)\oplus\spn(1)$ & $28$ & $4$\\
FII & $F_4$ & $\son(9)$ & $16$ & $1$\\
G & $G_2$ & $\son(4)$ & $8$ & $2$.\\
\hline
\end{tabular}
\end{center}

The spaces of group type are
$(L\times L)/\operatorname{diag}L\cong L$, where $L$ is a simply
connected compact simple Lie group with a bi-invariant metric. Their
list is as follows; the entries in the last five rows have no free
parameter.

\begin{center}
\small
\renewcommand{\arraystretch}{1.2}
\begin{tabular}{lllll}
\hline
Type & $L$ & Dimension & Rank & Parameters\\\hline
$A_\ell$ & $SU(\ell+1)$ & $\ell(\ell+2)$ & $\ell$ & $\ell\ge1$\\
$B_\ell$ & $\operatorname{Spin}(2\ell+1)$
 & $\ell(2\ell+1)$ & $\ell$ & $\ell\ge2$\\
$C_\ell$ & $Sp(\ell)$
 & $\ell(2\ell+1)$ & $\ell$ & $\ell\ge3$\\
$D_\ell$ & $\operatorname{Spin}(2\ell)$
 & $\ell(2\ell-1)$ & $\ell$ & $\ell\ge4$\\
$G_2$ & $G_2$ & $14$ & $2$ & \\
$F_4$ & $F_4$ & $52$ & $4$ & \\
$E_6$ & $E_6$ & $78$ & $6$ & \\
$E_7$ & $E_7$ & $133$ & $7$ & \\
$E_8$ & $E_8$ & $248$ & $8$ &\\
\hline
\end{tabular}
\end{center}
\begin{remark}
The lower bounds on $\ell$ merely remove repetitions:
$Sp(1)\cong\operatorname{Spin}(3)\cong SU(2)$,
$Sp(2)\cong\operatorname{Spin}(5)$, and
$\operatorname{Spin}(6)\cong SU(4)$.
In particular, the group-type row uses the simply connected groups
$\operatorname{Spin}(n)$, rather than $SO(n)$.
\end{remark}

\begin{proposition}\label{prop:fixed-rank-reduction}
Fix $r\ge1$ and an integer $B\ge2r$. Outside the three families
in \eqref{eq:families} with $N>B$, there are only finitely many
homothety classes of simply connected irreducible compact symmetric
spaces of rank $r$.
\end{proposition}

\begin{proof}
For types AI, AII, CI, and DIII, the rank determines or bounds the
parameter $n$. The same holds for the group-type spaces. There are
only finitely many exceptional spaces. In types AIII, BDI, and CII,
the smaller parameter is $p=r$, and imposing $N=p+q\le B$ leaves
only finitely many possibilities for $q$. Finally, an irreducible
compact symmetric space has a unique invariant symmetric metric
up to scale. This proves the assertion.
\end{proof}

\subsection{Isometry groups of Grassmannian manifolds}
For the real family we assume $N>\max\{8,2r\}$. For the complex
and quaternionic families it is enough to assume $N>2r$; in the
quaternionic rank-one case we also require $N\ge3$, excluding
$\HH P^1\cong S^4$. These are the hypotheses for the isometry-group
descriptions in this subsection. The stronger uniform threshold
$N>\max\{28,2r\}$ in \cref{thm:tail} also excludes the exceptional
quaternionic symmetric spaces that enter the proof of
\cref{thm:quartic-rigidity}.
We use the notation in \eqref{eq:families}.
In \cite{WolfGrass}, Wolf's descriptions give the following
surjective group homomorphisms with finite kernels:
\begin{align*}
Sp(N)&\longrightarrow\Isom(X_{\HH,r,N})=PSp(N),\\
 SU(N)\rtimes\langle\tau\rangle
 &\longrightarrow\Isom(X_{\C,r,N}),\\
 O(N)\times\langle\omega\rangle
 &\longrightarrow\Isom(X_{\R,r,N}).
\end{align*}
Here, $\tau$ is complex conjugation, acting on $SU(N)$ by
$a\mapsto\bar a$, and $\omega$ reverses the orientation of every
real $r$-plane. The first kernel is $\{\pm I\}$, and the second is
the finite scalar subgroup of $SU(N)$. In the third map an element
$(a,\omega^\epsilon)$ first applies the ambient orthogonal map $a$
and then reverses the plane orientation if $\epsilon=1$.
If $r$ is odd, $-I$ already induces $\omega$, and the kernel is
$\{(I,1),(-I,\omega)\}$, where $1$ denotes the identity of
$\langle\omega\rangle$. If $r$ is even, the kernel is
$\{(I,1),(-I,1)\}$.

Let $\rho:\widehat K\to\Isom(X_{\F,r,N})$ be the corresponding
map, and let $G$ be a compact isometry group. Its full preimage
$\widehat G=\rho^{-1}(G)$ is a compact matrix group, with the same
orbits on the Grassmannian as $G$.
If $G$ is finite, then
$\widehat G$ is finite because $\ker\rho$ is finite. In the complex
case its elements are unitary or antiunitary and hence real
orthogonal transformations. In the real case, forgetting orientation
leaves exactly the action of the projection of $\widehat G$ to
$O(N)$. A distance estimate for these underlying-plane orbits is a
lower bound for the oriented quotient distance. These observations
permit the finite-group estimate to be applied to real orthogonal
representations without any assumption about the components
of the original isometry group.

If $\dim G\geq 1$, we put
$\gfr=\Lie(G^0)$ and identify this Lie algebra with its lift through
$\rho$.
Since $G^0$ is normal in $G$, every element of the lifted group
normalizes $\gfr$. These conjugations preserve the ambient trace
scalar products. This observation will give invariance under the
whole group, including its disconnected components.

\section{Metric and analytic estimates }\label{sec:analytic}

Throughout this section, $\F$ is $\R$, $\C$, or $\HH$, and
$0<r<N$. Quaternionic vector spaces $V=\HH^N$ are taken to be right vector spaces, with
Hermitian product $\langle x,y\rangle_{V}=x^*y
=\sum_{i=1}^N\overline{x_i}y_i.$
This product is quaternion-valued. We use the real Riemannian scalar product
$\Rea(x^*y)$ on the underlying real vector space $\HH^N\cong\mathbb R^{4N}$. If an endomorphism
is quaternionic-linear, its real trace is four times the real part
of its quaternionic trace. In particular, $\Rea\tr_\HH(XY)
=\Rea\tr_\HH(YX)$, although the quaternionic trace itself need
not be cyclic.

Fix an $r$-plane $P\subset\F^N$. If an $r$-plane $Q$ satisfies
$Q\cap P^\perp=\{0\}$, then the restriction
$\pr_P|_Q:Q\to P$ is an isomorphism. Thus each $v\in P$ has a
unique lift to $Q$, which we write as $v+Lv$, with $Lv\in P^\perp$.
The resulting map $L:P\to P^\perp$ is $\F$-linear, and
\[
 Q=\operatorname{graph}(L)=\{v+Lv:v\in P\}.
\]
These graph charts identify the tangent space with
$\Hom_\F(P,P^\perp)$. We equip it with the standard metric
\begin{equation}\label{eq:metric}
 T_P\Gr_r(\F^N)=\Hom_\F(P,P^\perp),\qquad
 g_P(A,B)=\Rea\tr_\F(A^*B).
\end{equation}
If $e_1,\ldots,e_r$ is an orthonormal basis of $P$, then
$|A|_{g_P}^2=\sum_{j=1}^r|Ae_j|^2$. On the oriented real
Grassmannian we use the pullback metric under \eqref{eq:forgetmap}.

We shall also use the projector model. For a plane $W$, let $P_W$
be its orthogonal projector. Then
\[
 \Gr_r(\F^N)\cong
 \{P:P^*=P,\ P^2=P,\ \tr_\F P=r\}.
\]
Conversely, a matrix with these properties is the orthogonal
projector onto its image, which has dimension $r$. Relative to
$\F^N=W\oplus W^\perp$, a tangent map
$L\in\Hom_\F(W,W^\perp)$ corresponds to
\[
 \dot P_L=\begin{pmatrix}0&L^*\\L&0\end{pmatrix}.
\]
In a matrix expression such as $PAP$, the letter $P$ denotes the
projector onto the plane under consideration.

\begin{definition}\label{def:principal-angles}
For two $r$-planes $P,Q\subset\F^N$, let
$T=\pr_P|_Q:Q\to P$ be orthogonal projection restricted to $Q$.
Its singular values satisfy $1\ge s_1\ge\cdots\ge s_r\ge0$.
The principal angles are
\begin{align*}
\theta_j=\arccos s_j,\qquad
 0\le\theta_1\le\cdots\le\theta_r\le\pi/2.
\end{align*}
For real oriented planes, this definition uses their underlying
unoriented subspaces.
\end{definition}

The principal vectors can be chosen as follows. Diagonalize the
Hermitian operator $T^*T$ on $Q$, and choose an orthonormal basis
$q_1,\ldots,q_r$ such that
\[
 T^*Tq_j=\cos^2\theta_j\,q_j.
\]
When $\cos\theta_j>0$, set $p_j=Tq_j/\cos\theta_j$.
These vectors are an orthonormal basis of $\operatorname{im}T$.
Complete them to an orthonormal basis of $P$ for the indices with
$\theta_j=\pi/2$. We then have
\[
 \ip{p_i}{q_j}=\delta_{ij}\cos\theta_j.
\]
If $\theta_j=0$, then $p_j=q_j$. For each positive angle, define
\[
 n_j=\frac{q_j-\cos\theta_j\,p_j}{\sin\theta_j}\in P^\perp.
\]
The vectors $n_j$ are orthonormal, and $n_j=q_j$ when
$\theta_j=\pi/2$. The spectral theorem and singular-value
decomposition give these constructions over all three fields.

\begin{lemma}\label{lem:distance}
For the unoriented Grassmannian with metric \eqref{eq:metric},
\begin{align*}
d(P,Q)=\left(\sum_{j=1}^r\theta_j^2\right)^{1/2},\qquad
 \cos\theta_r(P,Q)=
 \min_{\substack{x\in Q\\|x|=1}}|\pr_Px|,
\end{align*}
for any two $r$-planes $P,Q\subset\F^N$.
\end{lemma}

\begin{proof}
Let $x=\sum_{j=1}^rq_j a_j\in Q$ be a unit vector, where $a_j\in \F$. The orthogonality of the
projected principal vectors gives $\pr_Px=\sum_{j=1}^rp_ja_j\cos\theta_j$, so
\begin{align*}
|\pr_Px|^2=\sum_{j=1}^r|a_j|^2\cos^2\theta_j,
 \qquad \sum_{j=1}^r|a_j|^2=1.
\end{align*}
The smallest possible value is $\cos^2\theta_r$, attained at
$x=q_r$. This proves the second formula.

For the upper bound on distance, we can rotate each positive-angle
direction independently:
\begin{align*}
u_j(t)=\cos(t\theta_j)p_j+\sin(t\theta_j)n_j
 \quad(\theta_j>0),\qquad
 u_j(t)=p_j\quad(\theta_j=0).
\end{align*}
Let $P(t)$ be their $\F$-linear span, for $0\le t\le1$.
The vectors $u_j(t)$ are orthonormal, $P(0)=P$, and $P(1)=Q$.
Moreover, every $\dot u_j(t)$ is perpendicular to $P(t)$, so
\begin{align*}
g_{P(t)}(\dot P(t),\dot P(t))=|\dot P(t)|^2=\sum_{j=1}^r|\dot u_j(t)|^2=\sum_{j=1}^r \theta_j^2.
\end{align*}
The length of this curve is the right side of the first formula.
For the lower bound, consider an arbitrary piecewise smooth plane
curve $P(t)$ from $P$ to $Q$. Choose a moving orthonormal frame
$u_1(t),\ldots,u_r(t)$ with terminal frame $q_1,\ldots,q_r$ and
with $\dot u_j(t)\perp P(t)$. Such a horizontal frame exists. Indeed,
starting with any orthonormal frame matrix $E(t)$, the matrix
$\Gamma(t)=E(t)^*\dot E(t)$ is skew-Hermitian. Replacing $E(t)$ by
$E(t)R(t)$, where $\dot R=-\Gamma R$ with prescribed terminal
orthogonal or unitary value, removes its internal frame rotation.
This construction also works for quaternionic unitary matrices.
Let $v_j=u_j(0)\in P$. The initial frame need not be a principal
frame. Nevertheless,
\begin{align*}
\Rea\ip{v_j}{q_j}
 =\Rea\ip{v_j}{\pr_Pq_j}\le|\pr_Pq_j|=\cos\theta_j.
\end{align*}
Thus the real spherical length
$\ell_j=\int_0^1|\dot u_j(t)|\,dt$ is at least $\theta_j$.
Using \eqref{eq:metric} and the integral triangle inequality in
$\R^r$, we obtain
\begin{align*}
L(P(t))
 =\int_0^1\left(\sum_j|\dot u_j(t)|^2\right)^{1/2}\,dt
 \ge\left(\sum_j\left(\int_0^1|\dot u_j(t)|\,dt\right)^2\right)^{1/2}
 =\left(\sum_j\ell_j^2\right)^{1/2}
 \ge\left(\sum_j\theta_j^2\right)^{1/2}.
\end{align*}
Taking the infimum over all connecting curves proves the distance
formula. The argument also covers zero and right principal angles.
\end{proof}

\begin{remark}
   In the real case, \cref{lem:distance} is the standard distance formula in
\cite[Section~4.3, p.~337]{EAS}; see also
\cite[Section~5.1, equation~(5.3)]{BZA}.
\end{remark}
\begin{lemma}\label{lem:diameters}
For $r\ge1$ and $N\ge2r$, in the metric \eqref{eq:metric},
\begin{align}
\label{eq:diameters}
 \diam X_{\C,r,N}=\diam X_{\HH,r,N}=\frac\pi2\sqrt r,
 \qquad \diam X_{\R,r,N}\le\pi\sqrt r.
\end{align}
The unoriented real Grassmannian also has diameter $\pi\sqrt r/2$.
For any two oriented real planes, distance is at least the distance
between their underlying planes.
\end{lemma}

\begin{proof}
Since every principal angle is at most $\pi/2$,
\cref{lem:distance} gives the bound $\pi\sqrt r/2$ in each
unoriented case. Since $N\ge2r$, we can choose mutually orthogonal
$r$-planes. All their principal angles equal $\pi/2$, and the bound
is attained. This proves the three unoriented diameter formulas.

For the oriented real case, choose principal frames $p_j,q_j$ with
$\langle p_i,q_j\rangle=\delta_{ij}\cos\theta_j$.
By simultaneous sign changes of corresponding $p_j,q_j$, we may
assume that the $p$-frame has the prescribed orientation of $P$.
For $\theta_j>0$, set
\[
 n_j=\frac{q_j-\cos\theta_j\,p_j}{\sin\theta_j}.
\]
These are orthonormal vectors in $P^\perp$. Complete them to an
orthonormal $r$-frame $n_1,\ldots,n_r$ in $P^\perp$, assigning
the added vectors to the zero-angle indices. This is possible
because $N-r\ge r$.

If the principal $q$-frame has the prescribed orientation of $Q$,
take $\alpha_j=\theta_j$ for all $j$. Otherwise, take
$\alpha_1=\theta_1-\pi$ and $\alpha_j=\theta_j$ for $j>1$.
Then
\[
 q'_j=\cos\alpha_j\,p_j+\sin\alpha_j\,n_j
\]
is a positively oriented orthonormal frame of $Q$: in the second
case only $q_1$ is replaced by $-q_1$. This also works when
$\theta_1=0$, since $\alpha_1=-\pi$ gives $q'_1=-p_1=-q_1$.
In particular, the construction includes the case in which all
principal angles vanish but the two orientations are opposite.

For $0\le t\le1$, put
\[
 u_j(t)=\cos(t\alpha_j)p_j+\sin(t\alpha_j)n_j
\]
and give $P(t)=\operatorname{span}_{\R}\{u_1(t),\ldots,u_r(t)\}$
the orientation of this frame. The planes
$\operatorname{span}_{\R}\{p_j,n_j\}$ are mutually orthogonal,
so the frame is orthonormal and each $\dot u_j(t)$ is perpendicular
to $P(t)$. Thus \eqref{eq:metric} gives
\[
 |\dot P(t)|^2=\sum_{j=1}^r|\dot u_j(t)|^2
             =\sum_{j=1}^r\alpha_j^2.
\]
This path joins $P$ to $Q$ and has length
$\sqrt{\sum_j\alpha_j^2}\le\pi\sqrt r$, proving the oriented
upper bound. Finally, the forgetful map is a local isometry and
preserves curve length, so oriented distance is at least the
distance between the underlying planes.
\end{proof}

Fix an orthonormal frame $e_1,\ldots,e_r$ of $P$ and a unit vector
$u\in P^\perp$. A \emph{unit-speed rank-one Grassmannian rotation}
is a curve of the form
\begin{align}
\label{eq:rank-one-rotation}
 P(t)=\operatorname{span}_\F
 \{\cos t\,e_1+\sin t\,u,e_2,\ldots,e_r\},\qquad t\in\R.
\end{align}
The displayed frame specifies orientation in the real oriented case.
The tangent map sends the first vector to
$-\sin t\,e_1+\cos t\,u$ and annihilates the other frame vectors.
Its image has $\F$-dimension one and its squared norm in
\eqref{eq:metric} is one.
A real {trigonometric polynomial of degree at most $d$} is
a function
\begin{align*}
p(t)=a_0+\sum_{k=1}^{d}(a_k\cos kt+b_k\sin kt),
 \qquad a_k,b_k\in\R.
\end{align*}

For example, the projector along \eqref{eq:rank-one-rotation}
has entries of degree at most two. A quadratic expression in those
entries therefore has degree at most four.
Bernstein's inequality states that
$\|p'\|_{L^\infty(\R)}\le d\|p\|_{L^\infty(\R)}$;
see \cite[equation~(1.2), p.~30]{Arestov}.

\begin{lemma}\label{lem:oscillation}
Let $G\subset\Isom(X_{\F,r,N})$ be a compact subgroup, and $F:X_{\F,r,N}\to\R$ be a smooth $G$-invariant function whose
range is exactly $[-1,1]$. Suppose its restriction to every rotation
\eqref{eq:rank-one-rotation} is a trigonometric polynomial of degree
at most an integer $d\ge1$. Then
\begin{align*}
\|\nabla F\|\le d\sqrt r,\qquad
 \diam(X_{\F,r,N}/G)\ge\frac{2}{d\sqrt r},
\end{align*}
 where $r\ge1$ and $N\ge2r$.
\end{lemma}

\begin{proof}
Every unit rank-one tangent map is the velocity of a curve
\eqref{eq:rank-one-rotation}. Since $|F|\le1$ along the entire curve,
Bernstein's inequality gives $|\mathrm dF(A)|\le d$ in each such
direction $A$.
Note that an arbitrary unit tangent has a singular-value decomposition
\begin{align*}
A=\sum_{j=1}^k\sigma_j A_j,
 \qquad k\le r,\quad \sigma_j>0,\quad \sum_{j=1}^k\sigma_j^2=1,
\end{align*}
where $A_j$ are unit rank-one tangent maps. Consequently
\begin{align*}
|\mathrm dF(A)|\le d\sum_{j=1}^k\sigma_j
 \le d\sqrt{k}\left(\sum_{j=1}^k\sigma_j^2\right)^{1/2}
 \le d\sqrt r.
\end{align*}
By integration along curves, the same global Lipschitz bound holds for
$F$. Since $F$ is invariant, for any $p,q$ and $g\in G$,
\begin{align*}
|F(p)-F(q)|=|F(p)-F(gq)|\le d\sqrt r\,d(p,gq).
\end{align*}
It follows that $d_{X_{\F,r,N}/G}(Gp,Gq)\geq \frac{1}{d\sqrt r}|F(p)-F(q)|.$
Choose points where $F$ attains $1$ and $-1$. Their quotient
distance is at least $2/(d\sqrt r)$.
\end{proof}

\begin{lemma}\label{lem:subspace}
Let $0<r<N$. Let $H\subset O(N)$ if $\F=\R$, and let
$H\subset Sp(N)$ if $\F=\HH$. If $\F=\C$, let $H$ be a subgroup
of the group of unitary and antiunitary transformations of $\C^N$.
Assume that $H$ is compact. These transformations preserve
$\F$-linear subspaces and induce isometries of $\Gr_r(\F^N)$.
Suppose there is an $\F$-linear vector subspace
$W\subset\F^N$ such that
\begin{align*}
0<\dim_\F W<N,\qquad hW=W\quad\text{for every }h\in H.
\end{align*}
Then
\begin{align*}
\diam(\Gr_r(\F^N)/H)\ge\pi/2.
\end{align*}
If a compact group $G$ acts on $\Gr_r^+(\R^N)$ and its induced
action on underlying planes is the action of such an $H$, then
$\diam(\Gr_r^+(\R^N)/G)\ge\pi/2$ as well.
\end{lemma}

\begin{proof}
Put $d=\dim_\F W$. Since the transformations of $H$ are isometries and
preserve $W$, they also preserve $W^\perp$. An $r$-plane that splits
as $A\oplus B$ with $A\subset W$, $B\subset W^\perp$ can have
$\dim_\F A=a$ exactly when
\begin{align*}
a_-:=\max(0,r+d-N)\le a\le\min(r,d)=:a_+.
\end{align*}
These inequalities imply that $0\le a\le d$ and
$0\le r-a\le N-d$. Moreover $a_+>a_-$. Indeed, if $r+d\le N$, then
$a_-=0<a_+$; otherwise $r+d-N$ is strictly smaller than both $r$
and $d$, since $r<N$ and $d<N$.

Choose subspaces $A_\pm\subset W$ of dimensions $a_\pm$ and
$B_\pm\subset W^\perp$ of dimensions $r-a_\pm$, and put
$P_\pm=A_\pm\oplus B_\pm$. For each $h\in H$, the projection
map from $A_+$ to $hA_-$ has rank at most $a_-$ and therefore has
a nonzero kernel. Choose a unit vector $x_h$ in this kernel.
It is perpendicular to $hA_-$; it is also perpendicular to $hB_-$,
since $x_h\in W$ and $hB_-\subset W^\perp$. Hence
\begin{align*}
x_h\in P_+,\qquad \pr_{hP_-}x_h=0.
\end{align*}
The largest principal angle between $P_+$ and $hP_-$ is $\pi/2$.
By \cref{lem:distance}, $d(P_+,hP_-)\ge\pi/2$ for every $h$.
Taking the minimum over the group proves the quotient bound.
This argument only uses preservation of dimensions, orthogonality,
and subspaces, so it also permits antiunitary transformations.

For oriented planes, choose either orientation of $P_+$ and $P_-$.
Every translated underlying pair has the same lower bound just
proved. The orientation-forgetting map is a local isometry and therefore
does not increase distances. This finishes the proof.
\end{proof}

\section{The case of finite acting groups}\label{sec:finite}
In this section, we prove \cref{thm:main} under the assumption that $G$ is finite.
\begin{theorem}\label{thm:finite}
For every positive integer $r$, there exists $c_{\mathrm{fin},r}>0$
such that every finite isometry group $G$ of every
simply connected irreducible compact symmetric space $M$ of rank $r$
satisfies
\begin{align*}
\Delta(M,G)\ge c_{\mathrm{fin},r}.
\end{align*}
In particular, the assertion includes $r=1$.
\end{theorem}

\begin{proof}
We first obtain a quantitative separation estimate for a finite
orthogonal group. The measure formulation of Green's theorem
\cite[Definition~1.6 and Theorem~1.8]{Green} provides an absolute
constant $a>0$ with the following property. For every integer $m\ge1$
and every finite subgroup $H\subset U(m)$, there is a probability measure
$\mu$ on the complex unit sphere such that
\begin{align}
\label{eq:green}
 \int \max_{h\in H}|\ip{hv}{z}|^2\,d\mu(z)
 \le\frac1{1+a\log m}
\end{align}
for every unit vector $v\in\C^m$.
The measure may depend on $H$, but it is the same measure for every
$v$. This quantifier is essential when separating an entire plane.

Let $H\subset O(m)$ be a finite group and let $E\subset\R^m$ be a real
$s$-plane, with orthonormal basis $e_1,\ldots,e_s$. Consider each
real matrix $h$ as a unitary matrix on $\C^m$. By applying the inequality
\eqref{eq:green} to each $e_j$, we obtain
\begin{align*}
\int\sum_{j=1}^s\max_{h\in H}|\ip{he_j}{z}|^2\,d\mu(z)
 \le\frac{s}{1+a\log m}.
\end{align*}
Since $\mu$ is a probability measure, the minimum of this continuous
integrand on the compact sphere is at most its integral. Hence there
is one unit vector $z=u+iv$, with $u,v\in\R^m$, such that
\begin{align*}
 \sum_{j=1}^s\max_{k\in H}|\ip{ke_j}{z}|^2
 \le\frac{s}{1+a\log m}.
\end{align*}
Here $z$ is chosen for the whole sum and is fixed for the rest of
the argument. Since $|u|^2+|v|^2=1$, one of these real
vectors has squared norm at least $1/2$. Denote it by $w$ and put
$x=w/|w|$. For every $h\in H$,
\begin{align*}
|\pr_{hE}x|^2
 =\frac1{|w|^2}\sum_{j=1}^s|\ip{he_j}{w}|^2
 \le2\sum_{j=1}^s|\ip{he_j}{z}|^2
 \le2\sum_{j=1}^s\max_{k\in H}|\ip{ke_j}{z}|^2
 \le\frac{2s}{1+a\log m}.
\end{align*}
The first inequality uses that each real or imaginary part of
$\ip{he_j}{z}$ has modulus at most $|\ip{he_j}{z}|$.
We have obtained one real unit vector $x$ satisfying
\begin{align}
\label{eq:finite-separate}
 \max_{h\in H}|\pr_{hE}x|^2\le\frac{2s}{1+a\log m}.
\end{align}

 Fix an $\F$-linear $r$-plane $E\subset\F^N$ and a finite
ambient group $H$ acting on its Grassmannian. With
$\kappa=\dim_\R\F\in\{1,2,4\}$, the underlying real
representation has dimension $m=\kappa N$ and the real dimension
of $E$ is $s=\kappa r$. Every ambient element is real orthogonal,
including the antiunitary elements allowed in the complex case.
An $\F$-linear subspace and its Hermitian orthogonal complement
are also orthogonal complements for the real scalar product.
Thus the real projection in \eqref{eq:finite-separate} is exactly
the Hermitian projection needed in \cref{lem:distance}.

 Now, extend $x$ to an $\F$-orthonormal $r$-frame and let $Q$ be its
span. Since $x\in Q$ and $|x|=1$, for every $h\in H$,
\begin{align*}
\cos\theta_r(hE,Q)
 =\min_{\substack{y\in Q\\|y|=1}}|\pr_{hE}y|
 \le|\pr_{hE}x|
 \le\min\left\{1,\sqrt{\frac{2\kappa r}{1+a\log(\kappa N)}}\right\}.
\end{align*}
The function $\arccos$ is decreasing and distance is at least the
largest principal angle. Hence
\begin{align}
\label{eq:finite-angle}
 \diam(X_{\F,r,N}/H)\ge
 \arccos\min\left\{1,
 \sqrt{\frac{2\kappa r}{1+a\log(\kappa N)}}\right\}.
\end{align}
In the real oriented case, one can choose orientations on $E,Q$ and use the
distance comparison in \cref{lem:diameters}, so the same bound holds.
Let $N_0(r)>\max\{8,2r\}$ be an integer such that
$1+a\log N_0(r)\ge32r$. If $N\ge N_0(r)$, then
\begin{align*}
1+a\log(\kappa N)\ge32r\ge8\kappa r.
\end{align*}
The square root in \eqref{eq:finite-angle} is at most $1/2$, so
the absolute quotient diameter is at least $\pi/3$, uniformly in
$H$ and $\F$. Dividing by the common upper bound
$\pi\sqrt r$ from \cref{lem:diameters} gives
\begin{align*}
\Delta(X_{\F,r,N},H)\ge\frac1{3\sqrt r}.
\end{align*}
For a finite subgroup $G$ of the full isometry group, use the
finite-kernel covers of \cref{sec:reduction}. Its full inverse
image is finite. In the complex and quaternionic cases it has
exactly the same Grassmannian orbits; in the real case its image in
$O(N)$ describes the underlying-plane orbits, and the orientation
comparison still applies. Thus the preceding estimate holds for
every finite intrinsic isometry group in this parameter range.

By \cref{prop:fixed-rank-reduction}, the remaining spaces form
finitely many homothety classes at fixed $r$. Choose
representatives $M_1,\ldots,M_t$. By \cref{thm:fixed-space}, each
has a number $\delta_j>0$ bounding below every positive quotient
diameter. A finite group cannot act transitively on a
positive-dimensional connected manifold. Therefore all the finite-group
quotients here have positive diameter. The positive number
\begin{align*}
c_{\mathrm{fin},r}=
 \min\left\{\frac1{3\sqrt r},
 \frac{\delta_1}{\diam M_1},\ldots,
 \frac{\delta_t}{\diam M_t}\right\}
\end{align*}
proves the theorem.
\end{proof}

\section{The real, complex and quaternionic Grassmannians}\label{sec:real-complex}

In this section, we prove \cref{thm:tail} under the assumption that $\dim G\geq 1$.

The real and complex arguments use the following result.
If $H$ acts by symplectomorphisms
on a symplectic manifold $(M,\omega)$, the fundamental vector field associated to
$X\in\h$ is $X^{\#}_p=\left.\frac{d}{dt}\right|_{t=0}\exp(tX)p$.
After identifying $\h^*$ with $\h$ by an invariant scalar product,
an equivariant \emph{moment map} is a smooth map $\mu:M\to\h$
satisfying
\begin{align*}
\mu(hp)=\Ad(h)\mu(p),\qquad
 d\ip{\mu}{X}=\iota_{X^{\#}}\omega\quad(X\in\h).
\end{align*}
Here $(\iota_{X^{\#}}\omega)(Y)=\omega(X^{\#},Y)$.
An action admitting such a map is called \emph{Hamiltonian}.

\begin{theorem}[{\cite[Theorem~1]{GP}}]
\label{thm:GP}
Let a compact connected Lie group $H$ act in a Hamiltonian fashion
by holomorphic isometries on a compact connected K\"ahler manifold $M$.
Choose an $\Ad(H)$-invariant positive scalar product on its Lie
algebra $\h$, and use it to identify an equivariant moment map with
a map $\mu:M\to\h$. If $\|\mu\|^2$ is constant, then $M$ is
$H$-equivariantly biholomorphic and isometric to a product of a
homogeneous flag manifold and a compact K\"ahler manifold on which
$H$ acts trivially.
\end{theorem}

In particular, a nontrivial, nontransitive Hamiltonian action on a
Riemannian irreducible compact K\"ahler manifold has nonconstant
moment-map norm. Indeed, nontriviality makes the flag factor
positive-dimensional. Irreducibility then makes the other factor a
point, so a constant norm would force transitivity. The theorem
applies without a rank restriction; in particular, it applies to
complex projective spaces. A related result in the symplectic
setting is proved in \cite{Biliotti}.

\begin{lemma}\label{lem:real-subsets}
Let $2\le r\le N-2$, and let $b_{ij}=b_{ji}$ be real numbers for
distinct $i,j\in\{1,\ldots,N\}$. Suppose that
\begin{align}\label{eq:Sb}
\sum_{\substack{i<j; ~i,j\in S}}b_{ij}
\end{align}
has the same value for every $r$-element subset $S$ of
$\{1,\ldots,N\}$. Then all $b_{ij}$ are equal.
\end{lemma}

\begin{proof}
Fix distinct indices $i,j$ and put $d_k=b_{ik}-b_{jk}$ for
$k\notin\{i,j\}$. By comparing the sums (\ref{eq:Sb}) for $T\cup\{i\}$ and
$T\cup\{j\}$, we obtain
\begin{align}
\label{eq:real-subset-difference}
 \sum_{k\in T}d_k=0
\end{align}
for every $(r-1)$-element subset $T$ avoiding $i,j$.
If $k,\ell$ are two distinct remaining indices, choose an
$(r-2)$-element set $U$ avoiding $i,j,k,\ell$. This is possible
because $0\le r-2\le N-4$; for $r=2$, take $U=\varnothing$.
Apply \eqref{eq:real-subset-difference} to $U\cup\{k\}$ and
$U\cup\{\ell\}$ and subtract. It follows that $d_k=d_\ell$.
Thus all $d_k$ have a common value $d$. By
\eqref{eq:real-subset-difference}, we have $(r-1)d=0$, so $d=0$.
Consequently $b_{ik}=b_{jk}$ whenever $i,j,k$ are distinct. All
edges incident to a fixed index therefore have the same value, and
an edge joining two indices identifies their common values. This
proves equality for every pair.
\end{proof}

\begin{remark}\label{rem:real-subsets-rank-one}
The restriction $r\ge2$ is essential. At $r=1$ every sum in the
hypothesis is empty and equals zero, independently of the numbers
$b_{ij}$. Thus the hypothesis imposes no condition. For example,
$b_{12}=1$ and all other $b_{ij}=0$ provide a counterexample when
$N\ge3$. The real-family argument below will retain $r\ge2$.
\end{remark}

\subsection{The real Grassmannians}
We equip $\son(N)$ with the invariant scalar product
\begin{equation}\label{eq:real-trace}
 \ip{X}{Y}_{\son(N)}=-\frac12\tr(XY),
 \qquad X,Y\in\son(N),
\end{equation}
and use it to identify the Lie algebra with its dual.
For an oriented two-plane $P$, choose a positively oriented
orthonormal basis $x,y$. The bivector $x\wedge y$ is independent
of that choice. Under the $SO(N)$-equivariant identification
\[
 \Lambda^2\R^N\longrightarrow\son(N),\qquad
 x\wedge y\longmapsto yx^T-xy^T,
\]
it corresponds to the matrix $Z_P=yx^T-xy^T$.

Let $J_Px=y$ and $J_Py=-x$. On the tangent space
$\Hom_\R(P,P^\perp)$, define
\[
 \omega_P(U,V)=g_P(-UJ_P,V).
\]
To verify that this is a K\"ahler structure, we identify the oriented
Grassmannian $\Gr_2^+(\R^N)$ with the complex quadric by
\begin{align*}
\Psi(P)=[x+iy]\in Q^{N-2}, \qquad
 Q^{N-2}=\{[z]\in\C P^{N-1}:z^Tz=0\}.
\end{align*}
This map is independent of the positively oriented orthonormal basis
$x,y$. At the unit representative $z=(x+iy)/\sqrt2$, its differential
has horizontal lift
\begin{align*}
\mathrm d\Psi_P(U)=\frac{Ux+iUy}{\sqrt2},\qquad
 \mathrm d\Psi_P(-UJ_P)=i\,\mathrm d\Psi_P(U).
\end{align*}
For the Fubini--Study metric of holomorphic sectional curvature $4$,
$\Psi^*g_{\mathrm{FS}}=g/2$. Thus $U\mapsto-UJ_P$ is integrable and
$\omega=2\Psi^*\omega_{\mathrm{FS}}$ is its K\"ahler form.

\begin{lemma}\label{lem:real-moment}
The map $\mu:\Gr_2^+(\R^N)\to\son(N)$ given by
$\mu(P)=Z_P$ is an equivariant moment map for $SO(N)$.
For a Lie subgroup $H\subset SO(N)$ with Lie algebra $\h$,
the map $\mu_H=\pr_\h\mu$ is an equivariant moment map for $H$.
\end{lemma}

\begin{proof}
Every tangent vector at $P$ has the form $Y^\#_P$, with
$Y\in\son(N)$. Differentiate the identity
\[
 Z_{\exp(tY)P}=\exp(tY)Z_P\exp(-tY)
\]
to obtain
\[
 d\ip{\mu}{X}_P(Y^\#_P)
 =\ip{[Y,Z_P]}{X}_{\son(N)}
 =\ip{Z_P}{[X,Y]}_{\son(N)}.
\]
Write the off-diagonal blocks of $X$ and $Y$, relative to
$\R^N=P\oplus P^\perp$, as
\[
 X=\begin{pmatrix}X_{11}&-U^T\\U&X_{22}\end{pmatrix},\qquad
 Y=\begin{pmatrix}Y_{11}&-V^T\\V&Y_{22}\end{pmatrix}.
\]
Their fundamental vectors are $U$ and $V$. Since $X_{11}$ and
$Y_{11}$ lie in the one-dimensional algebra $\son(P)$, they
commute. The upper-left block of $[X,Y]$ is therefore
$-U^TV+V^TU$. As $Z_P|_P=J_P$, we obtain
\begin{align*}
 \ip{Z_P}{[X,Y]}_{\son(N)}
 &=\frac12\tr\bigl(J_P(U^TV-V^TU)\bigr)\\
 &=\tr(J_PU^TV)
 =g_P(-UJ_P,V)=\omega_P(U,V).
\end{align*}
This proves the moment-map identity. Equivariance follows from
$Z_{aP}=aZ_Pa^{-1}$ for $a\in SO(N)$.
For $X\in\h$, orthogonal projection gives
$\ip{\mu_H}{X}=\ip{\mu}{X}$, so the same differential identity
holds for $H$. The projection is $H$-equivariant because the
scalar product is invariant.
\end{proof}

\begin{lemma}
\label{lem:real-quadric}
Let $N>8$, let $H\subset SO(N)$ be a compact and connected subgroup, and suppose
that its Lie algebra satisfies $0\ne\h\ne\son(N)$.
The squared norm of the projected moment map is
\begin{align*}
\psi_\h(x\wedge y)=\|\pr_\h(x\wedge y)\|_{\son(N)}^2,
\end{align*}
where $\{x,y\}$ is a positively oriented orthonormal basis of $P\in \Gr_2^+(\R^N)$. Then $\psi_\h$ is nonconstant.
\end{lemma}

\begin{proof}
The $H$-action is nontrivial, since the $SO(N)$-action on $\Gr_2^+(\R^N)$ has only
a finite kernel, and $\h\ne0$. The space $\Gr_2^+(\R^N)$ is Riemannian
irreducible for $N>8$. If $\psi_\h$ were constant,
then \cref{thm:GP} would make $H$ transitive on $\Gr_2^+(\R^N)$.
We now apply the compact transitivity classification in
\cite[p.~2484, equations~(1.2)--(1.5)]{WolfFlags}; see also \cite{On}.
Its hypotheses hold for this presentation. Modulo its finite
kernel, $SO(N)$ is the full connected group of Hermitian isometries
of the quadric; its Lie algebra is simple; and
$SO(2)\times SO(N-2)$ is the centralizer of the circle rotating
the distinguished two-plane. The image of $H$ is a compact
connected subgroup. In that classification the proper transitive
presentations of quadrics occur in ambient dimensions $7$ and $8$,
apart from smaller-dimensional identifications. The eight-dimensional case is listed explicitly in
\cite[p.~2484, equation~(1.5)]{WolfFlags}. Thus, for $N>8$, transitivity forces
$\h=\son(N)$, which contradicts the hypothesis. So $\psi_\h$ is nonconstant.
\end{proof}

\begin{lemma}
\label{lem:real-energy}
Let $N>8$ and $2\le r\le N-2$. Let $H\subset SO(N)$ be a compact and
connected subgroup with $0\ne\h\ne\son(N)$, and choose a basis
$A_1,\ldots,A_s$ of $\h$ orthonormal for
\eqref{eq:real-trace}. For a real rank-$r$ orthogonal projector
$P$, put
\begin{align}
\label{eq:real-energy}
 q_r(P)=\sum_{\nu=1}^s\|PA_\nu P\|_{\mathrm{Frob}}^2
 =-\sum_{\nu=1}^s\tr(PA_\nu PA_\nu).
\end{align}
Then $q_r$ is nonconstant. It is independent of the chosen
orthonormal basis and invariant under every ambient orthogonal
transformation normalizing $\h$. Its pullback to oriented planes
is also invariant under orientation reversal.
\end{lemma}

\begin{proof}
The equality in \eqref{eq:real-energy} follows from $P^2=P$, $P^T=P,$
$A_\nu^T=-A_\nu$, and cyclicity of trace. Note that orthogonal changes of
the basis of $\h$ preserve the sum of squared norms. If an
orthogonal transformation normalizes $\h$, then its conjugation action
preserves \eqref{eq:real-trace} and carries an orthonormal basis
of $\h$ to another one. This proves the stated invariance.
Orientation reversal leaves the projector unchanged.

Suppose that $q_r$ is constant, and fix an arbitrary
orthonormal basis $e_1,\ldots,e_N$ of $\R^N$. Set
\begin{align*}
b_{ij}=\sum_{\nu=1}^s\ip{A_\nu e_i}{e_j}^2=b_{ji}
 \qquad (i\ne j).
\end{align*}
Since $A_\nu$ is skew-symmetric, every diagonal entry is zero. Thus, on the
coordinate $r$-plane $P_S$,
\begin{align*}
q_r(P_S)=2\sum_{\substack{i<j;~i,j\in S}}b_{ij}.
\end{align*}
By \cref{lem:real-subsets}, all $b_{ij}$ are equal. Their total
sum is
\begin{align*}
\sum_{i<j}b_{ij}
 =\sum_{\nu=1}^s\ip{A_\nu}{A_\nu}_{\son(N)}=s,
\end{align*}
so their common value is $2s/[N(N-1)]$, which is independent of the
orthonormal basis. Every orthonormal pair extends to such a basis;
consequently
\begin{align}
\label{eq:real-pair-energy}
 \sum_{\nu=1}^s\ip{A_\nu x}{y}^2
 =\frac{2s}{N(N-1)}
 \quad (|x|=|y|=1,\ x\perp y).
\end{align}
Since
$\ip{A_\nu}{x\wedge y}_{\son(N)}=\ip{A_\nu x}{y}$ by \eqref{eq:real-trace},
the left side of \eqref{eq:real-pair-energy} is precisely $\psi_\h(x\wedge y)$.
This contradicts
\cref{lem:real-quadric}.
\end{proof}

\begin{theorem}\label{thm:real-tail}
 Let
$G\subset \Isom(X_{\R,r,N})$ be a compact and nontransitive subgroup, with
positive-dimensional identity component. Then
\begin{align*}
\diam(X_{\R,r,N}/G)\ge\frac1{\sqrt r},
 \qquad \Delta(X_{\R,r,N},G)\ge\frac1{\pi r},
\end{align*}
for $r\ge2$ and $N>\max\{8,2r\}$.
\end{theorem}

\begin{proof}
We use the finite-kernel ambient lift from \cref{sec:reduction}.
Let $H\subset SO(N)$ be the identity component of that lift. Its
Lie algebra $\h$ is nonzero. If $\h=\son(N)$, then $H$ would
be transitive on the oriented $r$-plane Grassmannian, making $G$
transitive as well. Hence $\h$ is proper, and
\cref{lem:real-energy} supplies a nonconstant $q_r$.
The entire lifted group normalizes $\h$. Together with the
orientation-reversal clause of \cref{lem:real-energy}, this makes $q_r$
invariant under all of $G$.

We verify its degree using an orthonormal moving frame. Along a
rank-one rotation, take
\begin{align*}
e_1(t)=\cos(t)e_1+\sin(t)u,\qquad
 e_j(t)=e_j\ (2\le j\le r),\qquad u\perp P,
\end{align*}
where $u$ is unit. Then
\begin{align*}
q_r(P(t))=2\sum_{\nu=1}^s\sum_{1\le a<b\le r}
          \ip{A_\nu e_a(t)}{e_b(t)}^2.
\end{align*}
The terms with $a,b\ge2$ are constant and every other pairing is
linear in $\cos t,\sin t$. Thus the trigonometric degree is at
most two for every $r\ge2$. Affinely normalize the range to
$[-1,1]$. By \cref{lem:oscillation}, the quotient diameter is at
least $1/\sqrt r$. The real diameter bound in
\eqref{eq:diameters} gives the normalized estimate.
\end{proof}

\subsection{The complex Grassmannians}

We equip $\sun(N)$ with the invariant scalar product
\begin{align}
\label{eq:complex-trace}
 \ip{X}{Y}_{\sun(N)}=-\Rea\tr(XY)
\end{align}
for $X,Y\in\sun(N)$. As before, $P$ denotes both a plane and its
orthogonal projector.
\begin{lemma} \label{lem:moment}
Let $1\le r<N$, and equip $\Gr_r(\C^N)$ with
the metric \eqref{eq:metric}. Define
\begin{align*}
J_r(P)=i\left(P-\frac rN I\right).
\end{align*}
With K\"ahler form $\omega(U,V)=g(iU,V)$,
the map $\mu=-J_r/2$ is a moment map for $SU(N)$.
\end{lemma}

\begin{proof}
The matrix $J_r(P)$ is skew-Hermitian and traceless, so it lies
in $\sun(N)$. At a plane $P$, decompose
$\C^N=P\oplus P^\perp$. A tangent map
$U\in\Hom_\C(P,P^\perp)$ corresponds, in the projector representation, to
\begin{align*}
\dot P_U=\begin{pmatrix}0&U^*\\ U&0\end{pmatrix}.
\end{align*}
Write an arbitrary $X\in\sun(N)$ as
\begin{align*}
X=\begin{pmatrix}X_{11}&-Z^*\\ Z&X_{22}\end{pmatrix}.
\end{align*}
Its infinitesimal action on the Grassmannian is the tangent map
$X^\#_P=Z$. Since
$\omega_P(Z,U)=\operatorname{Im}\tr(Z^*U)$, direct differentiation
gives
\begin{align*}
\mathrm d\ip{\mu}{X}_{\sun(N)}(U)
 =\frac12\Rea\tr(i\dot P_U X)
 =\frac12\Rea\bigl(i\tr(U^*Z-Z^*U)\bigr)
 =\operatorname{Im}\tr(Z^*U)
 =\omega_P(X^\#_P,U).
\end{align*}
Thus $\mu$ is a moment map. The equivariance follows from
$J_r(aPa^{-1})=aJ_r(P)a^{-1}$ for any $a\in SU(N)$.
\end{proof}

\begin{lemma} \label{lem:complex-moment}
If a compact connected subgroup $H\subset SU(N)$ acts nontrivially
and nontransitively on $\Gr_r(\C^N)$, then
\begin{align}
\label{eq:complex-energy}
 q_\h(P)=\|\pr_\h J_r(P)\|_{\sun(N)}^2,
 \qquad \h=\Lie(H),
\end{align}
is nonconstant. It is invariant under every unitary or antiunitary
ambient transformation whose conjugation action normalizes $\h$.
These conclusions include $r=1$.
\end{lemma}

\begin{proof}

The map $\mu_H=\pr_\h\mu$ is an equivariant moment map for $H$. The complex Grassmannian is compact, K\"ahler, and
Riemannian irreducible for every $1\le r<N$. If $q_\h$ were
constant, the identity $J_r=-2\mu$ would make
$\|\pr_\h\mu\|^2$ constant. The consequence of
\cref{thm:GP} would then force transitivity,
contrary to the hypothesis.

For a unitary transformation $a$, the conjugation on $\sun(N)$
preserves \eqref{eq:complex-trace}. If it normalizes $\h$,
then the orthogonal projection to $\h$ commutes with that conjugation.
This proves the invariance. If $a$ is antiunitary, let
$\alpha_a(X)=aXa^{-1}$ denote its real-linear conjugation action
on $\sun(N)$. It is again orthogonal for
\eqref{eq:complex-trace}, but $aia^{-1}=-i$ gives
\begin{align*}
J_r(aPa^{-1})=-\alpha_a(J_r(P)).
\end{align*}
If $\alpha_a(\h)=\h$, the projection commutes with $\alpha_a$ and
the sign disappears after taking the squared norm. This proves
the full invariance claim.
\end{proof}

\begin{theorem}\label{thm:complex-tail}
Let $r\ge1$ and $N>2r$. If
$G\subset\Isom(X_{\C,r,N})$ is a compact and nontransitive subgroup, with
positive-dimensional identity component, then
\begin{align*}
\diam(X_{\C,r,N}/G)\ge\frac1{2\sqrt r},
 \qquad \Delta(X_{\C,r,N},G)\ge\frac1{\pi r}.
\end{align*}
In particular, the normalized estimate is $1/\pi$ when $r=1$.
\end{theorem}

\begin{proof}
Lift $G$ through $SU(N)\rtimes C_2$ as in
\cref{sec:reduction}. Let $H$ be the identity component of
the lift. Then $H\subset SU(N)$ is compact and connected. Its
action is nontrivial since the lift has finite kernel and
$\dim G>0$. It is nontransitive because its orbits are
contained in $G$-orbits. By \cref{lem:complex-moment},
$q_\h$ is nonconstant. Every element of the lifted group normalizes $\h$. The unitary
and antiunitary invariance in \cref{lem:complex-moment} therefore
shows that $q_\h$ is invariant under all of $G$.

Along a unit rank-one rotation, write
\begin{align*}
P(t)=P_0+v(t)v(t)^*,\qquad
 v(t)=\cos(t)e+\sin(t)u,
\end{align*}
where $P_0$ is the fixed rank-$(r-1)$ projector and $e,u$ are
orthogonal unit vectors perpendicular to its range. The entries
of $P(t)$ are trigonometric polynomials of degree at most two.
Since $J_r$ is affine in $P$ and \eqref{eq:complex-energy} is
quadratic in $J_r$, its degree is at most four. For $r=1$ the
same formula holds with $P_0=0$.
Normalize the range of $q_\h$ to $[-1,1]$ and apply
\cref{lem:oscillation}. We conclude that the absolute bound is
$2/(4\sqrt r)=1/(2\sqrt r)$. Division by the complex
Grassmannian diameter $\pi\sqrt r/2$ proves the normalized
bound, including at rank one.
\end{proof}

Theorems~\ref{thm:real-tail} and~\ref{thm:complex-tail} prove the
real and complex cases of \cref{thm:tail}. The quaternionic case
requires a separate nonconstancy argument, which we give next.

\subsection{The quaternionic Grassmannians}\label{sec:quaternionic}

Write $V=\HH^N$ as a right quaternionic vector space, with the usual real
Euclidean scalar product $\ip{x}{y}_V=\Rea(x^*y)$. A quaternionic
matrix acts on the left. In this section, $\spn(N)$ is the
algebra of quaternionic skew-Hermitian matrices, equipped with
\begin{align}
\label{eq:trace}
 \ip{X}{Y}_{\spn(N)}=-\Rea\tr_\HH(XY),\qquad X,Y\in\spn(N).
\end{align}
This form is positive definite and invariant under conjugation by
$Sp(N)$. The restriction of \eqref{eq:trace} will always be used on
a subalgebra.

Let $\h\subset\spn(N)$ be a Lie subalgebra with $\dim \h=s$, and let
$A_1,\ldots,A_s$ be an orthonormal basis for the restriction (\ref{eq:trace}) to $\h$. Define
\begin{align}
\label{eq:quaternionic-energies}
 f(x)=\sum_{\nu=1}^s|x^*A_\nu x|^2,\qquad
 q_r(P)=\sum_{\nu=1}^s\|PA_\nu P\|_{\mathrm{Frob}}^2.
\end{align}
Here, $x\in V$, $P=P^*=P^2$ is the orthogonal projector onto an
$r$-dimensional quaternionic subspace, and
$\|Z\|_{\mathrm{Frob}}^2=\Rea\tr_\HH(Z^*Z)$. Thus $f$ is a real
homogeneous polynomial of degree four on $V$, and $q_r$ is a smooth
function on $\Gr_r(\HH^N)$. An orthogonal change of basis in $\h$
leaves both sums unchanged. Furthermore, any element of $Sp(N)$
normalizing $\h$ preserves these functions, since its adjoint action
preserves \eqref{eq:trace}.

Although the squared-norm expression for $q_r$ contains four copies
of $P$, by idempotence and cyclicity of the real trace we have
\begin{align*}
\|PA_\nu P\|_{\mathrm{Frob}}^2=-\Rea\tr_\HH(PA_\nu P A_\nu P)
   =-\Rea\tr_\HH(PA_\nu P A_\nu),
\end{align*}
Thus
\begin{align}
\label{eq:quadratic}
 q_r(P)=-\sum_{\nu=1}^s\Rea\tr_\HH(PA_\nu P A_\nu).
\end{align}
Consequently, $q_r$ is quadratic in the real coordinates of $P$.
For a unit vector $x$, the rank-one projector is $P=xx^*$ and
$PA_\nu P=x(x^*A_\nu x)x^*$. It follows that
\begin{align}
\label{eq:quaternionic-rank-one}
 q_1(xx^*)=f(x).
\end{align}
This identity will treat rank one directly.

We begin with the linear algebra that allows a quartic identity to
determine a four-form. Choose orthogonal complex structures $I,J,K$
in the algebra of right scalar multiplications satisfying
$I^2=J^2=K^2=-1$ and $IJ=K$. For example, with the usual quaternionic
units one may take $Ix=xi$, $Jx=xj$, and $Kx=-xk$ for any $x\in V$; this sign accounts
for the reversal of order in composition of right multiplications.
The group of right unit scalars will be denoted by $Sp(1)_{\mathrm R}$.
It commutes with the left action of $\spn(N)$.
For $L=I,J,K$, define the alternating two-form
$\eta_L(v,w)=\ip{Lv}{w}_V$ and the four-form
\begin{align*}
\Omega=\eta_I\wedge\eta_I+\eta_J\wedge\eta_J+\eta_K\wedge\eta_K.
\end{align*}
Note that right unit scalars rotate the three two-forms through $SO(3)$, so
$\Omega$ is $Sp(1)_{\mathrm R}$-invariant.

\begin{lemma}\label{lem:quaternionic-evaluation}
An $Sp(1)_{\mathrm R}$-invariant real alternating four-form on $V$
is determined by its values on $(x,Ix,Jx,Kx)$, for $x\in V$.
More precisely, the linear map
\begin{align*}
T:(\Lambda^4V^*)^{Sp(1)_{\mathrm R}}
   \longrightarrow (\Sym^4V^*)^{Sp(1)_{\mathrm R}},
 \qquad (T\Phi)(x)=\Phi(x,Ix,Jx,Kx),
\end{align*}
is an isomorphism. We identify a symmetric tensor of degree four
with its homogeneous quartic polynomial.
\end{lemma}

\begin{proof}
We first prove that $(T\Phi)(xq)=(T\Phi)(x)$ for every $x\in V$
and every unit quaternion $q$. Let $R_qx=xq$. The map
$L\mapsto R_q^{-1}LR_q$ preserves
$\operatorname{span}_{\R}\{I,J,K\}$ and its Euclidean scalar
product. Denote its matrix in the ordered basis $(I,J,K)$ by $O_q$.
Then $O_q\in SO(3)$: it is orthogonal, depends continuously on $q$,
and equals the identity at $q=1$, while $Sp(1)$ is connected.
Define
\begin{align*}
\widetilde I=R_q^{-1}IR_q,\qquad
 \widetilde J=R_q^{-1}JR_q,\qquad
 \widetilde K=R_q^{-1}KR_q.
\end{align*}
Then $(\widetilde I,\widetilde J,\widetilde K)=(I,J,K)O_q$. Note that $\Phi(x,\widetilde Ix,\widetilde Jx,\widetilde Kx)
 =\det(O_q)\Phi(x,Ix,Jx,Kx).$
By the invariance of $\Phi$ we obtain
\begin{align*}
\begin{aligned}
 (T\Phi)(R_qx)
 &=\Phi(R_qx,IR_qx,JR_qx,KR_qx)\\
 &=\Phi(x,\widetilde Ix,\widetilde Jx,\widetilde Kx)
 =(T\Phi)(x).
\end{aligned}
\end{align*}
Since $T\Phi$ is a homogeneous quartic polynomial, the map
$T$ is well-defined.

The complexification of the quaternionic representation has the form
\begin{align*}
V_{\mathbb C}\cong E\otimes H,
\qquad
\dim_{\mathbb C}E=2N,\quad H=\mathbb C^2.
\end{align*}
The complexified right $Sp(1)$-action is the standard $SL_2(\mathbb C)$-action on $H$, and it acts trivially on the space $E$. Since $H^*\cong H$ as an $SL_2$-module, we can write $V_{\mathbb C}^*\cong E^*\otimes H.$ The symmetric and exterior Cauchy formulas give
\begin{align*}
\operatorname{Sym}^4(E^*\otimes H)
 &=\bigoplus_{\lambda\vdash4}S_\lambda E^*\otimes S_\lambda H,\\
\Lambda^4(E^*\otimes H)
 &=\bigoplus_{\lambda\vdash4}S_\lambda E^*\otimes S_{\lambda'}H.
\end{align*}
Here, $\lambda$ is a partition of $4$, $\lambda'$ is its transpose, and $S_\lambda$ denotes the corresponding Schur functor. These are the standard Cauchy identities; see \cite[\S~6.2.8]{SamSnowden}. For a partition $(a,b)$,
\begin{align*}
S_{(a,b)}H
\cong
\operatorname{Sym}^{a-b}H\otimes(\det H)^b.
\end{align*}
On $SL_2$, the determinant factor is trivial. Thus this representation has an invariant vector exactly when $a=b$. In degree four, $a+b=4$, so the only possibility is $(a,b)=(2,2).$ Therefore
\begin{align*}
\left((\operatorname{Sym}^4V^*)^{Sp(1)}\right)_{\mathbb C}
\cong S_{(2,2)}E^*, \qquad \left((\Lambda^4V^*)^{Sp(1)}\right)_{\mathbb C}
\cong S_{(2,2)}E^*.
\end{align*} This Schur module is irreducible under $GL(E)$.
Every $g\in GL(N,\mathbb H)$ commutes with $I,J,K$. Using the usual action on covariant tensors,
\begin{align*}
\begin{aligned}
T(g\cdot\Phi)(x)
&=\Phi(g^{-1}x,g^{-1}Ix,g^{-1}Jx,g^{-1}Kx)\\
&=\Phi(g^{-1}x,I g^{-1}x,J g^{-1}x,K g^{-1}x)
=(T\Phi)(g^{-1}x).
\end{aligned}
\end{align*}
Hence, $T$ is $GL(N,\mathbb H)$-equivariant.
The Lie algebra of this group satisfies
\begin{align*}
\mathfrak{gl}(N,\mathbb H)\otimes_{\mathbb R}\mathbb C
\cong\mathfrak{gl}(2N,\mathbb C)=\mathfrak{gl}(E).
\end{align*}
Thus, $T_{\mathbb C}$ intertwines the $\mathfrak{gl}(E)$-actions on the two irreducible modules above. For each $L=I,J,K$, the quaternionic relations give
\begin{align*}
(\eta_L\wedge\eta_L)(x,Ix,Jx,Kx)=2|x|^4.
\end{align*}
Consequently,
\begin{align}
\label{eq:evalOmega}
 (T\Omega)(x)=\Omega(x,Ix,Jx,Kx)=6|x|^4.
\end{align}
This is not the zero polynomial, so $T_{\mathbb C}\neq0$. Schur's lemma now makes the complexified map an isomorphism, and
therefore the original real map is an isomorphism as well.
\end{proof}

\begin{theorem}\label{thm:quartic-rigidity}
Suppose $N>28$ and $0\ne\h\subset\spn(N)$ is a Lie subalgebra. If the quartic $f$ in
\eqref{eq:quaternionic-energies} is constant on the unit sphere of
$V$, then $\h=\spn(N)$.
\end{theorem}

\begin{proof}
We first derive a four-form identity from the constancy of $f$.
We then use this identity to construct a compact quaternionic
symmetric space. The classification reduces the final step to two
explicit calculations.

\medskip\noindent\emph{Step 1. Positivity and irreducibility.}
By homogeneity, $f(x)=a|x|^4$ for some $a\ge0$. In fact, $a>0$.
If $a=0$, then $x^*A_\nu x=0$ for every $x$ and every $\nu$.
In particular, $\ip{A_\nu x}{Ix}_V=0$ for every $x$. The real
endomorphism $A_\nu I=IA_\nu$ is symmetric, and
$\ip{A_\nu I x}{x}_V=-\ip{A_\nu x}{Ix}_V=0$.
This gives $A_\nu I=0$, hence, $A_\nu=0$ for every $\nu$,
contradicting $\h\ne0$.
The constancy hypothesis also forces $\h$ to act quaternionically
irreducibly. Indeed, suppose $V=W\oplus W^\perp$ is a proper
orthogonal splitting into $\h$-invariant quaternionic subspaces,
and choose unit vectors $y\in W$ and $z\in W^\perp$. Each $A_\nu$
is block diagonal, so for every unit quaternion $u$,
\begin{align*}
(y+zu)^*A_\nu(y+zu)
   =y^*A_\nu y+\bar u(z^*A_\nu z)u.
\end{align*}
The two diagonal values are imaginary quaternions. For any imaginary quaternion $q$,
\begin{align*}
\int_{\mathrm{Sp}(1)}\bar uqu\,du=0,
\end{align*}
where $du$ is Haar probability measure.
Therefore integration with respect to Haar probability measure
gives
\begin{align*}
\int_{Sp(1)}f(y+zu)\dd u=f(y)+f(z)=2a.
\end{align*}
On the other hand, $|y+zu|^2=2$, so the integrand is identically
$4a$. This contradiction proves irreducibility.

\medskip\noindent\emph{Step 2. The four-form identity.}
For $A\in\h$, define the alternating two-form
$\omega_A(v,w)=\ip{Av}{w}_V$. We use the wedge convention
\begin{align*}
(\omega\wedge\omega)(v_0,v_1,v_2,v_3)
 =2\bigl(\omega(v_0,v_1)\omega(v_2,v_3)
 -\omega(v_0,v_2)\omega(v_1,v_3)
 +\omega(v_0,v_3)\omega(v_1,v_2)\bigr).
\end{align*}
Since $A$ commutes with $I,J,K$, we obtain
\begin{align}
\label{eq:evalA}
 (\omega_A\wedge\omega_A)(x,Ix,Jx,Kx)=-2|x^*Ax|^2.
\end{align}
Indeed, let $b=\ip{Ax}{Ix}_V$, $c=\ip{Ax}{Jx}_V$,
and $d=\ip{Ax}{Kx}_V$. The other three paired values are
\begin{align*}
\omega_A(Jx,Kx)=-b,\qquad
 \omega_A(Ix,Kx)=c,\qquad
 \omega_A(Ix,Jx)=-d.
\end{align*}
Their contribution is $-2(b^2+c^2+d^2)$, and
$b^2+c^2+d^2=|x^*Ax|^2$. Note that each $\omega_A$ is invariant under
$Sp(1)_{\mathrm R}$, and so is $\Omega$. Hence
\begin{align*}
\Phi:=\sum_{\nu=1}^s\omega_{A_\nu}\wedge\omega_{A_\nu}
       +\frac a3\Omega
\end{align*}
is an invariant four-form. By
\eqref{eq:evalOmega} and \eqref{eq:evalA},
$T\Phi(x)=-2f(x)+2a|x|^4=0$. By
\cref{lem:quaternionic-evaluation}, we obtain the following identity
\begin{align}
\label{eq:bianchi}
 \sum_{\nu=1}^s\omega_{A_\nu}\wedge\omega_{A_\nu}
 +\frac a3\bigl(\eta_I\wedge\eta_I+
                 \eta_J\wedge\eta_J+\eta_K\wedge\eta_K\bigr)=0.
\end{align}

\medskip\noindent\emph{Step 3. Construction of a compact symmetric pair.}
Let $\mathfrak s=\operatorname{span}_\R\{I,J,K\}\simeq\spn(1)$
and $\kk=\h\oplus\mathfrak s$, acting on the real vector space $V$.
Choose the invariant positive scalar product on $\kk$ for which
the two summands are orthogonal, the restriction to $\h$ is
\eqref{eq:trace}, and
\begin{align*}
\ip{I}{I}_{\kk}=\ip{J}{J}_{\kk}=\ip{K}{K}_{\kk}=\frac3a,
 \qquad \ip{I}{J}_{\kk}=\ip{J}{K}_{\kk}=\ip{K}{I}_{\kk}=0.
\end{align*}
The action of $\kk$ on $V$ is faithful. To see this, suppose
$A+S=0$ as endomorphisms, with $A\in\h$ and $S\in\mathfrak s$.
Since $A$ commutes with $I,J,K$, so does $S=-A$. The center of
$\mathfrak s\cong\spn(1)$ is zero, and hence $S=A=0$.
The action is also real irreducible. Indeed, any real
$\kk$-invariant subspace is stable under $I,J,K$, hence quaternionic,
and the preceding argument proves $\h$-irreducibility.
We use the symmetric-pair criterion of Moroianu--Semmelmann
\cite[Proposition~2.5]{MS}, with its construction in
\cite[Lemma~2.1, Definition~2.2, and Remark~2.3]{MS}.
For a faithful real irreducible orthogonal representation of a
compact Lie algebra with a specified positive invariant scalar
product, that criterion says the following. If $Z_\alpha$ is an
orthonormal basis of the algebra and
$\omega_{Z_\alpha}(v,w)=\ip{Z_\alpha v}{w}_V$, then the bracket
construction below is a Lie bracket exactly when its
\emph{Casimir four-form} $\sum_\alpha
\omega_{Z_\alpha}\wedge\omega_{Z_\alpha}$ vanishes.
For our metric on $\kk$, an orthonormal basis is
\begin{align*}
A_1,\ldots,A_s,\quad \sqrt{a/3}\,I,\quad
 \sqrt{a/3}\,J,\quad \sqrt{a/3}\,K.
\end{align*}
Thus its Casimir four-form is exactly the left side of
\eqref{eq:bianchi}. So all hypotheses of the criterion are satisfied.
Here is the bracket in our conventions.
Define $C:\Lambda^2V\to\kk$ by
\begin{align}
\label{eq:symmetric-bracket}
 \ip{C(v,w)}{Z}_{\kk}=\ip{Zv}{w}_V\qquad(Z\in\kk),
\end{align}
and set $[Z,v]=Zv$ and $[v,w]=C(v,w)$ on
$\mathfrak l=\kk\oplus V$, keeping the original bracket on $\kk$.
Clearly, $C$ is uniquely defined, and the skew-adjointness of $Z$
makes it alternating. The invariance of the scalar products makes $C$
a $\kk$-equivariant map, so Jacobi identities with an entry in $\kk$
hold. For the remaining identity, an orthonormal basis $Z_\alpha$
as above gives
\begin{align*}
&\ip{C(u,v)w+C(v,w)u+C(w,u)v}{z}_V\\
&\qquad=\frac12\sum_\alpha
    (\omega_{Z_\alpha}\wedge\omega_{Z_\alpha})(u,v,w,z)=0.
\end{align*}
This is the Jacobi calculation in Proposition~2.5 of \cite{MS}.
It holds for either choice of sign in $[v,w]=\pm C(v,w)$.
The positive sign used here is the one for which the positive
scalar product
\[
 B=\ip{\cdot}{\cdot}_{\kk}\oplus\ip{\cdot}{\cdot}_V
\]
is invariant. Indeed, the mixed identity is
\begin{align*}
 B([v,w],Z)
 &=\ip{C(v,w)}{Z}_{\kk}
 =\ip{Zv}{w}_V\\
 &=-\ip{v}{Zw}_V=B(v,[w,Z]).
\end{align*}
All other invariance identities follow from the invariance of the
scalar product on $\kk$, the skew-adjointness of its action on $V$,
and the orthogonality of the two summands.
Moreover, $\sigma(Z+v)=Z-v$ is an involutive Lie algebra
automorphism. Thus $(\mathfrak l,\kk)$ is a symmetric pair with
a positive invariant scalar product.
The center of $\mathfrak l$ is zero. Indeed, it is preserved by
$\sigma$, and its $\kk$-component vanishes by faithfulness on $V$,
while its $V$-component is fixed by $\mathfrak s$ and therefore
vanishes. Consequently, $\mathfrak l$ is compact semisimple.
Let $L$ be its simply connected compact group and $K_0$ the identity
component of the subgroup fixed by the integrated involution.
The subgroup $K_0$ is closed, since it is the identity component
of the fixed subgroup of an automorphism. Thus $L/K_0$ is compact.
It is simply connected by the homotopy exact sequence of
$K_0\to L\to L/K_0$, since $L$ is simply connected and $K_0$ is
connected. The invariant metric induced from $V$ makes this quotient
a Riemannian symmetric space. It is irreducible because its isotropy
representation on $V$ is real irreducible.
Moreover, the holonomy algebra of this space is exactly $\kk$ in its action
on $V$. Indeed, $[V,V]$ spans $\kk$: if $Z\in\kk$ is orthogonal to
all $C(v,w)$, then \eqref{eq:symmetric-bracket} gives
$\ip{Zv}{w}_V=0$ for all $v,w$. Faithfulness implies $Z=0$.

The curvature formula is $R(v,w)z=-C(v,w)z$; see
\cite[Proposition~3.1.2]{Gorodski}. In our sign convention it gives
\[
 \ip{R(v,w)w}{v}_V=\|C(v,w)\|_{\kk}^2\ge0.
\]
More explicitly, if $e_1,\ldots,e_{4N}$ is a real orthonormal basis
of $V$, then
\begin{align*}
 \operatorname{Ric}(v,v)
 &=\sum_i\|C(e_i,v)\|_{\kk}^2
 =\sum_\alpha|Z_\alpha v|^2\\
 &=\sum_{\nu=1}^s|A_\nu v|^2+a|v|^2
 \ge a|v|^2.
\end{align*}
Here we used the orthonormal basis of $\kk$ displayed above.
Since $a>0$, the constructed space has positive Ricci curvature.
Since its curvature is parallel, its values generate the holonomy
algebra, which is therefore $\kk$; see
\cite[Remark~2.1.8]{Gorodski}.
The connected holonomy group preserves the algebra
$\R\operatorname{Id}\oplus\mathfrak s$, since its Lie algebra
$\kk=\h\oplus\mathfrak s$ normalizes this algebra. As $L/K_0$
is simply connected, full holonomy equals restricted holonomy.
Parallel transport therefore defines a parallel field of quaternion
algebras on $L/K_0$. Moreover, its scalar holonomy algebra is the
whole of $\mathfrak s$.
This is a quaternionic structure with \emph{quaternionic scalar
holonomy}, i.e., the scalar part of holonomy is the whole $Sp(1)$.

\medskip\noindent\emph{Step 4. The isotropy representation.}
We may now apply Wolf's classification
\cite[Theorem~5.4, pp.~1043--1044]{WolfOriginal}. Its explicit list
is also given in \cite[Proposition~2.1 and Table~1, pp.~266--267]{WolfTable}, which gives
the following possibilities for its quaternion-linear isotropy algebra.

\begin{center}
\begin{tabular}{ll}
Quaternion-linear isotropy algebra $\h$ & Quaternionic dimension $N$\\\hline
$\spn(N)$ & arbitrary\\
$\mathfrak u(N)$ & arbitrary\\
$\son(N)\oplus\spn(1)$ & arbitrary\\
$\spn(1),\ \spn(3),\ \sun(6),\ \son(12),\ \mathfrak e_7$
 & $2,\ 7,\ 10,\ 16,\ 28$, respectively.\\
 \hline
\end{tabular}
\end{center}
The first three spaces in this list are, respectively,
quaternionic projective space, the complex two-plane Grassmannian,
and the oriented real four-plane Grassmannian. We need their
isotropy representations, including the embeddings of the displayed
algebras in $\spn(N)$. These are obtained as follows.

For quaternionic projective space, the tangent space is $\HH^N$,
and the infinitesimal isotropy action is
\[
 (A,b)x=Ax-xb,\qquad A\in\spn(N),\quad b\in\spn(1).
\]
Within the isotropy algebra, the centralizer of the right scalar
factor is the standard
$\spn(N)$ acting on the left.

For the complex two-plane Grassmannian, the tangent space is
$M_{N,2}(\C)$. Its isotropy algebra consists of pairs
$(A,B)\in\mathfrak u(N)\oplus\mathfrak u(2)$ with
$\tr A+\tr B=0$, acting by $Z\mapsto AZ-ZB$.
Write $B=B_0+bI_2$, where $\tr B_0=0$. Then
$b=-\tfrac12\tr A$, and the action becomes
\[
 Z\longmapsto DZ-ZB_0,\qquad
 D=A+\tfrac12(\tr A)I_N.
\]
The map $A\mapsto D$ is an isomorphism of $\mathfrak u(N)$:
it is the identity on $\sun(N)$ and multiplication by
$1+N/2$ on the center. Under the real isometry
$(z,w)\mapsto z+w j$ from $M_{N,2}(\C)$ to $\HH^N$,
the $B_0$-factor gives right quaternionic scalars, and $D$ gives
the standard complex-matrix inclusion
$\mathfrak u(N)\subset\spn(N)$.

For the oriented real four-plane Grassmannian, the tangent space is
$\operatorname{Hom}_{\R}(\R^4,\R^N)\cong\R^N\otimes\HH$.
The $\son(N)$-factor acts on $\R^N$, and the two simple factors of
$\son(4)$ act by left and right imaginary quaternions on $\HH$.
After choosing the right factor as $\mathfrak s$, its centralizer
in the isotropy algebra is the real skew-matrix algebra together with scalar left imaginary
quaternions. This is the stated standard embedding of
$\son(N)\oplus\spn(1)$.

The five entries of the last row come from the compact spaces of
types $G_2,F_4,E_6,E_7,E_8$ in Wolf's table.

Let $\h'\subset\spn(N)$ be one of $\spn(N)$, $\mathfrak{u}(N)$,
$\son(N)\oplus\spn(1)$ in standard position and
$\kk'=\h'\oplus\mathfrak s$. We claim that {if $N>28$, there are $\psi\in Sp(N)$ and a standard $\h'$ with $\psi\h\psi^{-1}=\h'$, and then
$f_{\h'}(\psi x)=f_\h(x)$ for all $x$.}

To prove the claim, use Wolf's classification. Since $N>28$,
there is a homothety of $L/K_0$ onto one of the three classical
quaternionic symmetric spaces that preserves base points.
Multiplying its differential by a positive scalar gives a real
linear isometry $\varphi:V\to V$. Holonomy algebras correspond
under the differential, and scalar multiplication does not affect
conjugation. Hence $\varphi\kk\varphi^{-1}=\kk'$ for one of the
three $\kk'$.

First we arrange $\varphi\mathfrak s\varphi^{-1}=\mathfrak s$. The image $\varphi\mathfrak{s}\varphi^{-1}$ is an ideal
of $\kk'$ isomorphic to $\spn(1)$, hence one of the simple ideals of $\kk'$. For
$\kk'=\spn(N)\oplus\mathfrak{s}$ and $\kk'=\sun(N)\oplus\mathfrak{u}(1)\oplus\mathfrak{s}$, the only such ideal is $\mathfrak{s}$. For
$\kk'=\son(N)\oplus\spn(1)\oplus\mathfrak{s}$, there are two, $\spn(1)$ and $\mathfrak{s}$. Let
$c(x)=\bar x$ be coordinatewise conjugation on $\HH^N=\R^N\otimes\HH$. It is a linear isometry, commutes with
$\son(N)$, and satisfies $c(xq)=\bar q\,c(x)$, so it normalizes $\kk'$ and interchanges the two ideals.
Replace $\varphi$ by $c\circ\varphi$ if necessary.

Next we make the intertwiner quaternionic-linear. Conjugation by $\varphi$ maps the algebra
$\R\oplus\mathfrak{s}=\{R_q:q\in\HH\}$ of right multiplications onto itself, so $\varphi R_q\varphi^{-1}=R_{\alpha(q)}$
for a real algebra automorphism $\alpha$ of $\HH$. Indeed,
$R_pR_q=R_{qp}$, and conjugating this identity gives
$\alpha(qp)=\alpha(q)\alpha(p)$. Thus $\varphi(xq)=\varphi(x)\alpha(q)$. Every
automorphism of $\HH$ is inner, $\alpha(q)=uqu^{-1}$ with $|u|=1$. Then $\psi(x)=\varphi(x)u$ satisfies
$\psi(xq)=\psi(x)q$ and is an isometry, so $\psi\in Sp(N)$. Right multiplication by $u$ commutes with $\h'$
and normalizes $\mathfrak{s}$, so still $\psi\kk\psi^{-1}=\kk'$ and $\psi\mathfrak{s}\psi^{-1}=\mathfrak{s}$.

Now $\h$ is the centralizer of $\mathfrak{s}$ in $\kk$ (because $\kk=\h\oplus\mathfrak{s}$, $[\h,\mathfrak{s}]=0$ and $\mathfrak{s}$
has trivial center), and $\h'$ is the centralizer of $\mathfrak{s}$ in $\kk'$. Hence $\psi\h\psi^{-1}=\h'$.

Finally, the trace form (\ref{eq:trace}) is $\Ad(Sp(N))$-invariant, so $\psi A_\nu\psi^{-1}$ is an orthonormal basis of
$\h'$ when $A_\nu$ is one of $\h$, and $(\psi x)^*(\psi A_\nu\psi^{-1})(\psi x)=x^*A_\nu x$. Therefore
$f_{\h'}(\psi x)=f_\h(x)$.

The restricted trace form is therefore unchanged. We can now
compute in the standard matrix models.

\medskip\noindent\emph{Step 5. Exclusion of the two proper classical cases.}
Let $E_{bc}$ be the elementary matrix and choose
\begin{align*}
x=e_1,\qquad y=\frac{e_1+e_2j}{\sqrt2}\in V.
\end{align*}
For $\h=\mathfrak u(N)$, a trace-orthonormal basis is
\begin{align*}
iE_{bb},\qquad
 U_{bc}=\frac{E_{bc}-E_{cb}}{\sqrt2},\qquad
 V_{bc}=\frac{i(E_{bc}+E_{cb})}{\sqrt2}\quad(b<c).
\end{align*}
At $x$ only $iE_{11}$ contributes to $f$, with squared modulus one.
At $y$ the only nonzero values are
\begin{align*}
y^*(iE_{11})y=\frac i2,\qquad
 y^*(iE_{22})y=-\frac i2,\qquad
 y^*U_{12}y=\frac j{\sqrt2},\qquad
 y^*V_{12}y=\frac k{\sqrt2}.
\end{align*}
Consequently
\begin{align*}
f(x)=1,\qquad f(y)=\frac14+\frac14+\frac12+\frac12=\frac32.
\end{align*}
For $\h=\son(N)\oplus\spn(1)$, an orthonormal basis consists of
the $U_{bc}$ and
$iI_N/\sqrt N$, $jI_N/\sqrt N$, $kI_N/\sqrt N$.
The real skew matrices contribute zero at $x$, and the three
scalar matrices contribute $1/N$ each. At $y$, the sole real-skew
contribution is $|y^*U_{12}y|^2=1/2$. The scalar values are
\begin{align*}
y^*\frac{iI_N}{\sqrt N}y=0,\qquad
 y^*\frac{jI_N}{\sqrt N}y=\frac j{\sqrt N},\qquad
 y^*\frac{kI_N}{\sqrt N}y=0.
\end{align*}
Thus
\begin{align*}
f(x)=\frac3N,\qquad f(y)=\frac12+\frac1N.
\end{align*}
These values coincide only when $N=4$. Both proper classical
possibilities contradict constancy of $f$, so $\h=\spn(N)$.
\end{proof}

\begin{remark}
The cutoff $N>28$ belongs to the rigidity assertion about the
quartic and is independent of the plane rank. The exceptional
representations in Wolf's table cannot be discarded on the basis
of quartic constancy alone. The symmetric-space classification is
used only after the four-form identity \eqref{eq:bianchi} and the
quaternionic scalar holonomy have been established.
\end{remark}

\begin{lemma}\label{lem:quaternionic-bridge}
Let $1\le r<N$, let $N>28$ with $N\ne2r$, and let $G\subset Sp(N)$
be a compact group with $\gfr=\Lie(G^0)\ne0$. If $r\ge2$, assume that
the whole group $G$ acts quaternionically irreducibly on $V$.
Form $f$ and $q_r$ using a trace-orthonormal basis
$A_1,\ldots,A_s$ of $\h=\gfr$ as above.
If $q_r$ is constant on $\Gr_r(\HH^N)$, then
$\gfr=\spn(N)$. For $r=1$, no irreducibility assumption is required.
\end{lemma}

\begin{proof}
If $r=1$, it follows from \eqref{eq:quaternionic-rank-one} that
$f$ is constant on the unit sphere. So $\gfr=\spn(N)$
by \cref{thm:quartic-rigidity}.
Suppose now $r\ge2$. Consider the Casimir endomorphism
\begin{align*}
\mathcal C=-\sum_{\nu=1}^s A_\nu^2=\sum_{\nu=1}^s A_\nu^*A_\nu.
\end{align*}
The Casimir endomorphism is quaternionic Hermitian and positive
semidefinite. Every $g\in G$ normalizes $\gfr$ and preserves the
trace scalar product. Hence $gA_\nu g^{-1}$ is another orthonormal
basis of $\gfr$, and basis independence gives
\[
 g\mathcal Cg^{-1}=-\sum_\nu(gA_\nu g^{-1})^2
                  =\mathcal C.
\]
The eigenspaces of $\mathcal C$ have real eigenvalues and are
quaternionic subspaces invariant under the whole group $G$.
Thus the assumed irreducibility of $G$ gives $\mathcal C=cI_N$.
This step does not require the restriction to $G^0$ to be irreducible.
Taking the real part of the quaternionic trace shows that
$Nc=\sum_\nu\|A_\nu\|_{\mathrm{Frob}}^2=s$, so $c=s/N>0$.

Let $e_1,\ldots,e_N$ be an arbitrary quaternionic orthonormal basis of $\HH^N,$
 and put
\begin{align*}
a_i=\sum_\nu|(A_\nu)_{ii}|^2=f(e_i),\qquad
 b_{ij}=\sum_\nu|(A_\nu)_{ij}|^2=b_{ji}\quad(i\ne j).
\end{align*}
For an index set $S\subset\{1,\ldots,N\}$ of cardinality $r$,
let $P_S$ be projection onto its coordinate span. Then by (\ref{eq:quaternionic-energies}), we have
\begin{align}
\label{eq:coordinate-energy}
 q_r(P_S)=\sum_{i\in S}a_i+
       2\sum_{\substack{i<j;~i,j\in S}}b_{ij},\qquad
 c=a_i+\sum_{k\ne i}b_{ik}.
\end{align}
Fix distinct indices $i,j$. For every $(r-1)$-element subset $T$
avoiding $i,j$, the constancy of $ q_r$ on the two planes with index sets
$T\cup\{i\}$ and $T\cup\{j\}$ gives
\begin{align*}
0=a_i-a_j+2\sum_{k\in T}(b_{ik}-b_{jk}).
\end{align*}
Next, we average this identity over all such $T$. Each of the $N-2$ available indices
occurs in a proportion $(r-1)/(N-2)$ of the subsets. Moreover,
subtracting the two scalar-Casimir identities in
\eqref{eq:coordinate-energy} and cancelling $b_{ij}=b_{ji}$ gives
\begin{align*}
\sum_{k\ne i,j}(b_{ik}-b_{jk})=-(a_i-a_j).
\end{align*}
It follows that
\begin{align}
\label{eq:bridge-coefficient}
 0=\left(1-\frac{2(r-1)}{N-2}\right)(a_i-a_j)
   =\frac{N-2r}{N-2}(a_i-a_j).
\end{align}
Since $N\ne2r$, all the $a_i$ are equal. The basis was arbitrary,
so $f(u)=f(v)$ for any quaternionically orthogonal unit vectors.
Any two quaternionic lines have a common perpendicular line when
$N\ge3$; comparing each of their unit vectors with a unit vector
on this common perpendicular line proves that $f$ is constant on
the entire unit sphere. The conclusion now follows from
\cref{thm:quartic-rigidity}.
\end{proof}

\begin{remark}\label{rem:middle}
For $r\ge2$, the coefficient in \eqref{eq:bridge-coefficient}
vanishes at $N=2r$. This is a limitation of the subset-averaging
argument; no failure of rigidity is asserted there.
The case $N=2r$ is outside the range of \cref{thm:quaternionic-tail}.
At fixed rank it contributes only one homothety class, whose quotient
diameters are handled by \cref{thm:fixed-space} in the proof of
\cref{thm:main}.
Rank one uses \eqref{eq:quaternionic-rank-one} directly and has no
subset-averaging obstruction.
\end{remark}

\begin{theorem}\label{thm:quaternionic-tail}
Let $r\ge1$, $N>\max\{28,2r\}$. If
$G\subset \Isom(X_{\HH,r,N})$ is a compact nontransitive group, with
positive-dimensional identity component, then
\begin{align*}
\diam(X_{\HH,r,N}/G)\ge\frac1{2\sqrt r},\qquad
 \Delta(X_{\HH,r,N},G)\ge\frac1{\pi r}.
\end{align*}
In particular, the estimate \eqref{eq:tail} holds for quaternionic
Grassmannians in every rank $r\ge1$.
\end{theorem}

\begin{proof}
Since $N>\max\{28,2r\}$, the full isometry group of $X_{\HH,r,N}$ is $PSp(N)$.
Consider the inverse image of $G$ under the finite-kernel covering
$Sp(N)\to PSp(N)$, as in \cref{sec:reduction}, and denote it again
by $G$. The lift is compact, has positive-dimensional identity
component, and has exactly the same Grassmannian orbits. Put
$\gfr=\Lie(G^0)\ne0$. It is proper in $\spn(N)$, since the full
symplectic algebra would make the action transitive.

If $r=1$, \cref{lem:quaternionic-bridge} already says that $q_1$
is nonconstant. Suppose $r\ge2$. If the whole group preserves a
proper nonzero quaternionic subspace, \cref{lem:subspace} gives
$\diam(X_{\HH,r,N}/G)\ge\pi/2$, which is stronger than the
claimed bound. Otherwise the whole group is quaternionically
irreducible, and \cref{lem:quaternionic-bridge} makes $q_r$
nonconstant as well.

The conjugation by every component of $G$ preserves $\gfr$ and the trace
scalar product, so $q_r$ is invariant under the whole group.
Let $P$ be an arbitrary $r$-plane in $\HH^N$.
Consider a unit-speed rank-one Grassmannian rotation, which replaces one unit frame
vector $e\in P$ by $e(t)=\cos(t)e+\sin(t)u$, with $u\perp P$ unit.
Its projector is
$P(t)=P-ee^*+e(t)e(t)^*$, whose entries have trigonometric degree
at most two. Equation \eqref{eq:quadratic} therefore gives degree
at most four for $q_r(P(t))$, which is independent of $N$ and $\dim\gfr$.
Let $m$ and $M$ be the minimum and maximum of $q_r$. Since $m<M$,
the invariant function
\begin{align*}
F(P)=\frac{2q_r(P)-M-m}{M-m}
\end{align*}
has range $[-1,1]$ and the same trigonometric degree. Applying
\cref{lem:oscillation} with $d=4$ gives
$\diam(X_{\HH,r,N}/G)\ge 2/(4\sqrt r)$.
Finally, \eqref{eq:diameters} gives
$\diam X_{\HH,r,N}=\pi\sqrt r/2$, so division yields
$\Delta(X_{\HH,r,N},G)\ge1/(\pi r)$. This completes the proof.
\end{proof}

\section{Proof of the main results and examples}\label{sec:assembly}

We now combine the finite-group estimate with the estimates for
positive-dimensional groups.
\begin{proof}[Proof of \cref{thm:main}]
Replace the acting group by its closure, which leaves the quotient
metric unchanged, and omit the transitive case. We first give the
rank-one reduction mentioned in \cite{GLLM}.
Let $\eps_S>0$ be the constant in \cref{thm:sphere-gap}.
For a unit sphere, that theorem gives
$\Delta(S^n,G)\ge\eps_S/\pi$.

For $\mathbb F=\C$ or $\HH$, equip $\mathbb F P^m$, $m\ge2$,
with the quotient metric of the unit-sphere Hopf fibration.
Every isometry of $\C P^m$ is induced by a unitary or antiunitary
map on $\C^{m+1}$, and every isometry of $\HH P^m$ is induced
by an element of $Sp(m+1)$; see \cite[Sections~3--4]{WolfGrass}.
Let $\mathcal L_{\C}$ be the group of unitary and antiunitary
maps on $\C^{m+1}$. Let $\mathcal L_{\HH}$ be the orthogonal
group of transformations $x\mapsto Axq^{-1}$, where
$A\in Sp(m+1)$ and $q\in Sp(1)$. In each case there is a
surjective homomorphism
\[
 \rho:\mathcal L_{\mathbb F}\longrightarrow
 \Isom(\mathbb F P^m).
\]
Its kernel acts transitively on every Hopf fibre. Thus
$\widehat G=\rho^{-1}(G)$ is a compact group acting on the
unit sphere and containing the fibre action. This construction
uses an extension of $G$ and does not require a homomorphic
choice of lifts of its elements.

For unit vectors $x,y$, the distance formula is
\[
 d_{\mathbb F P^m}([x],[y])
 =\arccos|x^*y|
 =\min_{q\in\mathbb F,\ |q|=1}d_S(x,yq).
\]
Taking the infimum over $G$ shows that
$\mathbb F P^m/G$ and $S^{(\dim_{\R}\mathbb F)(m+1)-1}/\widehat G$
are isometric as metric spaces. Since the former has positive
diameter, \cref{thm:sphere-gap} applies to the latter.
As $\diam(\mathbb F P^m)=\pi/2$, this gives
$\Delta(\mathbb F P^m,G)\ge2\eps_S/\pi$.
The spaces $\C P^1$ and $\HH P^1$ are round spheres up to
homothety. Finally, choose one metric on the Cayley plane $M_O$.
By \cref{thm:fixed-space}, its positive quotient diameters are
bounded below by some $\delta_O>0$. The rank-one assertion
therefore holds with
\[
 \eps_1=\min\{\eps_S/\pi,\ \delta_O/\diam M_O\}>0.
\]

Now let $r\ge2$. On the three families of Grassmannians \eqref{eq:families} with
$N>\max\{28,2r\}$, the acting groups with positive-dimensional identity
component are covered by \cref{thm:tail}. A closed subgroup of the
isometry group is a compact Lie group. If its identity component is
zero-dimensional, the group is finite and
\cref{thm:finite} applies.
By \cref{prop:fixed-rank-reduction}, the remaining homothety
classes have finitely many representatives $M_1,\ldots,M_t$. Since each representative is compact homogeneous
and simply connected, \cref{thm:fixed-space} shows that there exists $\delta_j>0$
such that every positive-diameter isometric quotient of $M_j$ has
diameter at least $\delta_j$. Therefore
\begin{align*}
\eps_r=\min\left\{\frac1{\pi r},\ c_{\mathrm{fin},r},
 \frac{\delta_1}{\diam M_1},\ldots,
 \frac{\delta_t}{\diam M_t}\right\}>0
\end{align*}
has the required property. Every term is independent of the acting
group, and overall metric scalings leave the quotient ratio unchanged. This completes the proof.
\end{proof}

\begin{proof}[Proof of \cref{cor:cover}]
Let $p:\widetilde M\to M$ be the universal Riemannian covering and $\Gamma$ its deck group. Note that every isometry of $M$ lifts to an isometry of $\widetilde M$. Let $\widetilde G\subset\Isom(\widetilde M)$ be the group of all
lifts of elements of $G$. It contains $\Gamma$, and $p\circ\tilde g=g\circ p$ when $\tilde g$ lifts $g$.
For $\tilde x,\tilde y\in\widetilde M$ over $x,y\in M$ and $\tilde g$ over $g$, we have
$d_M(x,gy)=\min_{\gamma\in\Gamma}d_{\widetilde M}(\tilde x,\gamma\tilde g\tilde y)$, hence
\begin{align*}
    \inf_{g\in G}d_M(x,gy)=\inf_{\tilde g\in\widetilde G}d_{\widetilde M}(\tilde x,\tilde g\tilde y).
\end{align*}
So $M/G$ and $\widetilde M/\widetilde G$ are isometric, and $\diam(M/G)=\diam(\widetilde M/\widetilde G)$. If
this number is positive, Theorem~\ref{thm:main} gives
$\diam(M/G)\ge\eps_r\diam\widetilde M\ge\eps_r\diam M$, because the covering map $p$ does not increase distances.
\end{proof}

We finish with three examples that delimit the conclusion. The first
is the simply connected version of the mechanism in
\cite[Example~32]{GLLM}.

\begin{example}\label{ex:rank}
There is no positive lower bound independent of rank, even among
simply connected irreducible compact symmetric spaces.
For $r\ge2$, let $M_r=\operatorname{Spin}(2r+1)$ have the
bi-invariant metric pulled back from the scalar product
$-\tr(AB)/2$ on $\son(2r+1)$. This is a simply connected
irreducible compact symmetric space of rank $r$. Right multiplication
by $G_r=\operatorname{Spin}(2r)$ has quotient the unit sphere
$S^{2r}$, of diameter $\pi$.
On $SO(2r+1)$, the element with $r$ planar rotation blocks of angle
$\pi$ and one fixed coordinate lies at distance $\pi\sqrt r$ from
the identity. Indeed, in this bi-invariant metric the distance is
the least norm of a skew-symmetric logarithm, and these blocks have
minimal rotation angles $\pi$. A covering local isometry does not
increase distances under projection, so $\diam M_r\ge\pi\sqrt r$.
Consequently
\begin{align*}
0<\Delta(M_r,G_r)\le\frac1{\sqrt r}\longrightarrow0.
\end{align*}
Thus dependence on $r$ in \cref{thm:main} is necessary.
\end{example}

\begin{example}\label{ex:products}
If factors may be scaled independently, a fixed-rank gap fails for
reducible symmetric spaces. Let
\begin{align*}
M_t=S^2(1)\times S^2(t),\qquad t>0,
\end{align*}
where the arguments denote the sphere radii. Let $SO(3)$ act
transitively on the first factor and trivially on the second. The
space is simply connected and symmetric of rank two, but
\begin{align*}
\Delta(M_t,SO(3))=
 \frac{\pi t}{\sqrt{\pi^2+\pi^2t^2}}
 =\frac{t}{\sqrt{1+t^2}}\longrightarrow0.
\end{align*}
This explains the irreducibility hypothesis when only the total
diameter is normalized.

This example does not address the case in which every de Rham
factor has diameter $1$. A corresponding fixed-rank statement under
that normalization would require a further argument for actions
that couple or permute the factors.
\end{example}

\begin{example}\label{ex:rank-one-tail}
 Let $m\ge15$ and view the unit sphere $S^{2m-1}$ as the unit sphere
in the real matrix space $\R^{m\times2}$ with its Frobenius norm.
The connected group $SO(m)\times SO(2)$ acts by
$X\mapsto AXB^{-1}$. Its orbits are determined by the two singular values. To see why
special orthogonal matrices suffice, start with an ordinary singular-value
decomposition. A sign change in a right singular vector can be matched
by the same sign change in the corresponding left singular vector.
The determinant of the left matrix can then be corrected in the unused
orthogonal complement, whose dimension is $m-2>0$. Thus there exist matrices
$U\in {SO}(m)$ and $V\in{SO}(2)$ such that $ X=U\begin{pmatrix}
    \Sigma\\ O
    \end{pmatrix}V,$
where $\Sigma=\operatorname{diag}(\sigma_1,\sigma_2)$ and
$\sigma_1\geq \sigma_2\geq 0$.
Therefore, every orbit contains the matrix $X_\theta$ whose only
nonzero entries are $(X_\theta)_{11}=\cos\theta$ and
$(X_\theta)_{22}=\sin\theta$, with $0\le\theta\le\pi/4$.
The curve $\theta\mapsto X_\theta$ has speed one, so the quotient
diameter is at most $\pi/4$. For the reverse inequality, the orbit of
$X_0$ consists of unit rank-one matrices $ab^T$, with $|a|=|b|=1$.
Their inner products with $X_{\pi/4}$ satisfy
\begin{align*}
\ip{ab^T}{X_{\pi/4}}_{\mathrm{Frob}}
 =a^T X_{\pi/4}b\le\|X_{\pi/4}\|_{\mathrm{op}}=\frac1{\sqrt2},
\end{align*}
where $\|\cdot\|_{\mathrm{op}}$ denotes the operator norm of a matrix, i.e.,
$\|A\|_{\mathrm{op}}=\max_{|v|=1}|Av|.$
Their spherical distance from $X_{\pi/4}$ is therefore at least
$\pi/4$. The quotient has diameter exactly $\pi/4$, and
\begin{align*}
\Delta(S^{2m-1},SO(m)\times SO(2))=\frac14<\frac1\pi.
\end{align*}
Here $S^{2m-1}=X_{\R,1,2m}$ and $2m>28$. This example explains the
real rank restriction in \cref{thm:tail}.
\end{example}

\section*{Acknowledgements}
This paper is partially supported by NSFC (Grant Nos. 12571025, 12131012 and 12301032), NSF of Jiangsu (Grant
No. BK20230803), Fundamental Research Funds for the Central Universities (Grant No. 4007012402), and Zhishan Scholars Programs of Southeast University (Grant No. 2242024RCB0039).

\section*{Use of generative AI.}
During the preparation of this work, the authors used ChatGPT for exploratory computations, possible proof directions, and editorial polishing. The authors reviewed and verified all output.

\bigskip
\noindent
Zhiqi Chen\\
School of Mathematics and Statistics, Guangdong University of Technology,
Guangzhou 510520, P.R. China\\
E-mail: \href{mailto:chenzhiqi@nankai.edu.cn}{chenzhiqi@nankai.edu.cn}

\medskip\noindent
Shaoqiang Deng\\
School of Mathematical Sciences and LPMC, Nankai University,
Tianjin 300071, P.R. China\\
E-mail: \href{mailto:dengsq@nankai.edu.cn}{dengsq@nankai.edu.cn}

\medskip\noindent
Hui Zhang\\
School of Mathematics, Southeast University, Nanjing 211189, P.R. China\\
E-mail: \href{mailto:huizhang@mail.nankai.edu.cn}{huizhang@mail.nankai.edu.cn}
\end{document}